\documentclass[a4paper,12pt]{amsart}
\usepackage{amsmath}
\usepackage{amssymb}
\usepackage{amsfonts}
\usepackage{amscd}
\usepackage{mathrsfs}
\usepackage{paralist}
\usepackage{url}
\usepackage{enumerate}
\usepackage[all]{xy}
\usepackage[dvipdfmx]{graphicx}
\usepackage{graphics}
\usepackage{tikz}
\everymath{\displaystyle}
\theoremstyle{definition}
\newtheorem{defi}{\indent\bf Definition}[section]
\newtheorem{thm}{\indent\bf Theorem}[section] 
\newtheorem{lem}[thm]{\indent\bf Lemma}     
\newtheorem{cor}[thm]{\indent\bf Corollary}   
\newtheorem{prop}[thm]{\indent\bf Proposition} 

\newtheorem{remark}[thm]{\indent\bf Remark} 
\newtheorem{obs}[thm]{\indent\bf Observation}

\newtheorem{fact}[thm]{\indent\bf Fact}

\def\Re{\text{Re\,}} 
\def\Im{\text{Im\,}} 
\def\coloneqq{\mathrel{\mathop:}=}
\newcommand{\C}{\mathbb{C}}
\newcommand{\Q}{\mathbb{Q}}
\newcommand{\R}{\mathbb{R}}

\newcommand{\Z}{\mathbb{Z}}
\renewcommand{\H}{\mathbb{H}}

\newcommand{\del}{{\partial}}
\newcommand{\delbar}{\overline{\partial}}

\newcommand{\eps}{\epsilon}

\newcommand{\dist}{{\rm{dist}}}

\newcommand{\Ker}{{\rm{Ker}\,}}

\renewcommand{\i}{\sqrt{-1}}
\renewcommand{\tilde}{\widetilde}
\renewcommand{\bar}{\overline}
\renewcommand{\hat}{\widehat}
\renewcommand{\eps}{\varepsilon}
\renewcommand{\phi}{\varphi}

\begin{document}
\pagestyle{plain}

\title[On Bott--Chern and Aeppli cohomologies of two-dimensional toroidal groups]{On Bott--Chern and Aeppli cohomologies of two-dimensional toroidal groups}
\author[Jinichiro Tanaka]{Jinichiro Tanaka$^{1}$}
\address{
$^{1}$ Department of Mathematics \\
Graduate School of Science \\
Osaka Metropolitan University \\
3-3-138 Sugimoto \\
Osaka 558-8585 \\
Japan 
}
\email{sw23876x@st.omu.ac.jp}

\renewcommand{\thefootnote}{\fnsymbol{footnote}}
\renewcommand{\thefootnote}{\arabic{footnote}}


\begin{abstract} 
A toroidal group is a generalization of a complex torus, and is obtained as the quotient of the complex Euclidean space $\mathbb{C}^n$ by a discrete subgroup. 
Toroidal groups with finite-dimensional cohomology, called theta toroidal groups, are known to exhibit behavior analogous to that of complex tori. 
We compute Bott--Chern and Aeppli cohomologies for two-dimensional non-compact theta toroidal groups. 
\end{abstract}

\maketitle

\begin{section}{Introduction}\label{intro}
In this paper, we compute Bott--Chern and Aeppli cohomohologies for the two-dimensional non-compact theta toroidal groups. 
A {\it toroidal group} is a connected Abelian complex Lie group which has no non-constant holomorphic function. 
Every toroidal group is a complex manifold obtained as the quotient $\C^n/\Lambda$ of the complex Euclidian space $\C^n$ by a discrete subgroup $\Lambda \subset \C^n$. 
A complex torus is an example of compact toroidal groups. 

If all the cohomologies $H^{p,q}(X, \mathcal{O})$ are finite dimensional for a toroidal group $X$, $X$ is called a {\it theta toroidal group}; otherwise, $X$ is said to be a {\it wild toroidal group}. 
It is known that $H^{p,q}(X, \mathcal{O})$ is non-Hausdorff and infinite dimensional when $X$ is a wild toroidal group. 
This classification, based on whether all the dimensions of $ H^{p,q}(X, \mathcal{O})$ are finite or not, was established by Kazama \cite{Kazama}. 
Specifically, the Dolbeault cohomology groups of theta toroidal groups are explicitly calculated in \cite{Kazama}. 
Subsequently, Umeno specified the de Rham cohomology groups $H^k_{DR}(X)$ of toroidal groups, and proved the Hodge decomposition $\textstyle H^k_{DR}(X)=\bigoplus_{p+q=k} H^{p,q}(X, \mathcal{O})$ for theta toroidal groups \cite{Ume}. 
Moreover, Kazama and Takayama proved that the $\del\delbar$-lemma is valid for theta toroidal groups, but invalid for wild theta toroidal groups \cite{KaTa}. 
Generally, the $\del\delbar$-lemma is asserted for compact complex manifolds, but we state  it for toroidal groups: namely, a toroidal group $X$ is said to satisfy the $\del\delbar$-lemma if every $d$-exact form on $X$ is $\del\delbar$-exact.
Thus, theta toroidal groups are found to behave like complex tori. 

On the other hand, the Bott--Chern and the Aeppli cohomologies play important roles in the study of compact (non-K\"{a}hler) manifolds. 
This is due to the fact that while these two cohomologies coincide on compact K\"{a}hler manifolds, but they do not on compact non-K\"{a}hler manifolds. 
The $\del\delbar$-lemma is known not to always hold for general compact manifolds. 
However, Angella--Tomassini \cite{AT} and Angella--Taradini \cite{ATa} showed that for compact manifolds, the validity of the $\del\delbar$-lemma is characterized by the dimensions of these two cohomologies (see \S \ref{ddbar_lemma}). 

For these reasons, we focus on theta toroidal groups that satisfy the $\del\delbar$-lemma but are not compact, investigating the relationship between their Bott--Chern and Aeppli cohomologies. 
This paper examines Bott--Chern and Aeppli cohomologies of two-dimensional non-compact toroidal groups. 

Let $X$ be a complex manifold. 
We define {\it Bott--Chern} and {\it Aeppli cohomologies} as follows: 
\begin{align*}
H^{p,q}_{BC}(X) &\coloneqq
\left\{
\begin{array}{llll}
\frac{\Ker d\cap\mathcal{A}^{p,q}(X)}{\del\delbar\mathcal{A}^{p-1,q-1}(X)} &(p\geq1,q\geq1), \\
\frac{\Ker d\cap \mathcal{A}^{p,0}(X)}{\del \Omega^{p-1}(X)} &(p\geq1,q=0), \\
\frac{\Ker d\cap \mathcal{A}^{0,q}(X)}{\delbar\, \bar{\Omega}^{q-1}(X)} &(p=0,q\geq1),\\
\Ker d\cap \mathcal{A}(X) &(p=0,q=0),
\end{array}
\right. \\
H^{p,q}_A(X) &\coloneqq
\left\{
\begin{array}{llll}
\frac{\Ker \del\delbar\cap\mathcal{A}^{p,q}(X)}{\del\mathcal{A}^{p-1,q}(X)+\delbar\mathcal{A}^{p,q-1}(X)} &(p\geq1,q\geq1), \\
\frac{\Ker \del\delbar\cap \mathcal{A}^{p,0}(X)}{\del \mathcal{A}^{p-1,0}(X)+\Omega^p(X)} &(p\geq1,q=0), \\
\frac{\Ker \del\delbar\cap \mathcal{A}^{0,q}(X)}{\delbar\mathcal{A}^{0,q-1}(X)+\bar{\Omega}^q(X)} &(p=0,q\geq1),\\
\frac{\Ker \del\delbar\cap \mathcal{A}(X)}{\bar{\mathcal{O}}(X)+\mathcal{O}(X)} &(p=0,q=0),
\end{array}
\right. \\
\end{align*}
where $\mathcal{A}^{p,q}$ is the sheaf of germs of $C^\infty$ $(p,q)$-forms on $X$. 
Denote by $h_{BC}^{p,q}$ and $h_{A}^{p,q}$ their dimensions. 

\begin{remark}
The definitions of the Bott--Chern and Aeppli cohomology groups as above differ from the conventional ones. 
In particular, the differences from the standard definitions occur in the cases where $p=0$ or $q=0$. 
In contrast, when both $p$ and $q$ are at least $1$, our definitions agree with the usual ones. 
Our motivation for employing these definitions is to examine the extent of the gap between local and global solvability (see Lemma \ref{Sch_ex}). 
\end{remark}

For $\tau\in\H=\{z\in\C \mid \Im z >0\}$ and real numbers $p$ and $q$, we define the $\Z$-generated $\R$-independent discrete subgroup $ \Lambda_{\tau,p,q} \subset \C^2$ as follows:  
\begin{align}\label{lattice}
	\Lambda_{\tau,p,q} \coloneqq \left\langle 
	\begin{pmatrix}
	0 \\
	1 \\
	\end{pmatrix},\ 
	\begin{pmatrix}
	1 \\
	p \\
	\end{pmatrix},\ 
	\begin{pmatrix}
	\tau \\
	q \\
	\end{pmatrix}
	\right\rangle.
\end{align}
It is known that any two-dimensional non-compact theta toroidal group can be represented by the quotient $\C^2/\Lambda_{\tau,p,q}$ for a certain tuple $(\tau,p,q)$ (See \S $\ref{2-dim}$). 

For a two-dimensional non-compact theta toroidal group $X=\C^2/\Lambda_{\tau,p,q}$, we obtain the following theorems.
\begin{thm}[{Bott--Chern cohomology}]\label{BCcohom}
The Bott--Chern numbers of a two-dimensional non-compact theta toroidal group $X$ are:
\begin{align*}
\begin{array}{ccccc}
     &  & h^{0,0}_{BC}=1 &  &  \\
     & h^{1,0}_{BC}=2 &  & h^{0,1}_{BC}=2 &  \\
    h^{2,0}_{BC}=1 &  & h^{1,1}_{BC}=3 &  & h^{0,2}_{BC}=1 \\
     & h^{2,1}_{BC}=1 &  & h^{1,2}_{BC}=1 &  \\
     &  & \phantom{\;.} h^{2,2}_{BC}=0 \;. &  & 
\end{array}
\end{align*}
\end{thm}

\begin{thm}[{Aeppli cohomology}]\label{Aeppli}
The Aeppli numbers of a two-dimensional non-compact theta toroidal group $X$ are:
\begin{align*}
\begin{array}{ccccc}
     &  & h^{0,0}_{A}=1 &  &  \\
     & \tilde{h}^{1,0}_{A}=1 &  & \tilde{h}^{0,1}_{A}=1 &  \\
    h^{2,0}_{A}=0 &  & \tilde{h}^{1,1}_A=2 &  & h^{0,2}_{A}=0 \\
     & h^{2,1}_{A}=0 &  & h^{1,2}_{A}=0 &  \\
     &  & \phantom{\;.} h^{2,2}_{A}=0 \;, &  & 
\end{array}
\end{align*}
where $\tilde{h^{p,q}_A}$ are the dimensions of the Hausdorff completion of $H^{p,q}_A(X)$. 
Here, we consider $\mathcal{A}^{p,q}(X)$ with the topology of the uniform convergence of all the derivatives of coefficients on compact sets, and give $H^{p,q}_A(X)$ the quotient topology.
\end{thm}

\begin{remark}
A detailed description of the Bott--Chern and Aeppli cohomology groups with respect to the standard definitions is provided in Remark \ref{standard_def}.
\end{remark}

The paper is organized as follows. 
In \S \ref{preli}, we introduce the fundamentals of toroidal groups and the tools for computing several cohomologies. 
For a two-dimensional toroidal group $X$, we describe the construction of $X$ and present some known facts about it in \S \ref{2-dim}. 
To prove main theorems, we will use some sheaf exact sequences. 
Hence, we explain the special sheaves on toroidal groups in \S \ref{sheaves}, and the local resolution lemma in \S \ref{local}. 
In addition, we briefly introduce the computation methods and known results for the Dolbeault cohomology of theta toroidal groups in \S \ref{Dol}. 
In \S \ref{pf_11}, we prove Theorem \ref{BCcohom}. 
Then we obtain the representations of each Bott--Chern cohomology groups for theta toroidal groups. 
In \S \ref{sec_Aeppli}, we analyze a quotient vector space isomorphic to Aeppli cohomology for two-dimensional theta toroidal groups. 
In \S \ref{pf_12}, we prove Theorem \ref{Aeppli}. 
Finally, we present the fact about the $\del\delbar$-lemma for compact complex manifolds in \S \ref{ddbar_lemma}, and the special cohomology relevant to the local resolution lemma in \S \ref{thirdcohom}.

{\bf Acknowledgement. }
The author would like to thank Professor Hisashi Kasuya and Professor Giovanni Placini for thier invaluable guidance on the research policy, and Prof. Takayuki Koike for his helpful comments. 
We are also grateful to Professor Laurent Stolovitch for the insightful discussions regarding the covergence problem of formal power series. 
Furthermore, I am thankful to Professor Xiaojun Wu for providing critical comments that led to important corrections in the main results. 
This work was partly supported by MEXT Promotion of Distinctive Joint Research Center Program JPMXP0723833165 and Osaka Metropolitan University Strategic Research Promotion Project (Development of International Research Hubs). 
This work was also supported by JST SPRING Grant Number JPMJSP2139 and JSPS Bilateral Program Number JPJSBP120243210.

\end{section}
\begin{section}{Preliminary}\label{preli}
In this section, we introduce some fundamentals of toroidal groups and the tools necessary for the proof. 
In advance of each subsection, we will present facts about general theta toroidal groups. 

\begin{fact}[{\cite[Theorem $3.3$]{KaTa}}]\label{ddbar}
The $\del\delbar$-lemma holds on a theta toroidal group $X$. 
That is, for any $C^\infty$ $(k,\ell)$-form $\phi$ on $X$ with $k\geq1$ and $\ell\geq1$, $\phi$ is $d$-exact if and only if $\phi$ is $\del\delbar$-exact.
\end{fact}

\begin{fact}[{\cite[Proposition $2.1$]{Ume}}]\label{dexact}
Let $X$ be a toroidal group. 
For every $d$-closed $C^\infty$ $k$-form $\phi$ on $X$, there exist the $k$-form $\chi$ with constant coefficients on $X$ and a $C^\infty$ $(k-1)$-form $\psi$ on $X$ such that 
\begin{align}\label{d_repre}
\phi= \chi +d \psi.
\end{align}
Furthermore, $\chi$ is uniquely determined. 
\end{fact}

\begin{remark}
It is also known what happens to the assertions of Fact \ref{ddbar} and Fact \ref{dexact} on wild toroidal groups. 
\cite[Theorem $3.3$]{KaTa} states that the $\del\delbar$-lemma does not hold on a wild toroidal group. 
On the other hand, the validity of Fact \ref{dexact} does not depend on whether $X$ is a theta toroidal group or a wild toroidal group. 
\end{remark}

These facts will be invaluable when we subsequently compute cohomologies. 
Each subsection will now lay the groundwork for the proof of Theorem \ref{BCcohom} and Theorem \ref{Aeppli}.
\begin{subsection}{Two-dimensional toroidal groups}\label{2-dim}
In this subsection, we discuss the fundamentals of two-dimensional toroidal groups. 

We define a discrete subgroup $\Lambda_{\tau,p,q} \subset \C^2$ as in $(\ref{lattice})$ and the complex surface $X\coloneqq \C^2/\Lambda_{\tau,p,q}$.
The complex surface $X$ is a toroidal group if and only if either $p$ or $q$ is an irrational number, \cite{Kop} (See also \cite[Theorem $1.1.4$]{AK}). 
In other words, every two-dimensional non-compact toroidal group $X$ can be written in the form $\C^2/\Lambda_{\tau,p,q}$ for some element $\tau\in\H$ and some pair $(p,q)\in\R^2\setminus\Q^2$.
Let $\pi: X \to \C/\langle 1,\tau \rangle$ be the mapping induced by the first projection $\C^2\to \C$ ; $(z,w)\to z$. 
Note that $\pi: X \to \C/\langle 1,\tau \rangle$ can be regarded as a $\C^\ast$-bundle structure on the elliptic curve $\C/\langle 1,\tau \rangle$. 
 
Toroidal groups are classified into theta toroidal groups and wild toroidal groups based on number theoretical conditions, as shown in \cite{Kazama}. 
In particular, \cite[Theorem $4.3$]{Kazama} states that $X=\C^2/\Lambda_{\tau,p,q}$ is a theta toroidal group if and only if there exist $C>0$ and $0<\delta<1$ such that: 
\begin{align}\label{theta}
\dist (\Z^2, (np,nq)) \geq C\delta^n
\end{align}
for any $n\in\Z_{>0}$, where $\dist$ is the Euclidean distance (See also \cite[Section $2$]{KoTa}). 

Denote the standard basis of $\R^2$ by $\{e_1,e_2\}$. 
The coordinate transformation $\C^2 \ni (z_1,z_2) \mapsto (t_1,t_2,t_3,t_4)\in\R^4$ defined by
\begin{align*}
z_1e_1+z_2e_2=  t_1 (1,p)+t_2 (0,1) +t_3 (\tau, q)+t_4 \sqrt{-1}e_2
\end{align*}
is an isomorphism between real Lie groups. 
Then the following relations hold:
\begin{align*}
t_1&=\Re z_1 -\frac{\Re\tau}{\Im\tau} \Im z_1, & &t_2=\Re z_2 -p\left(\Re z_1+ \frac{\Re\tau}{\Im\tau}\right)- q\Im z_1,\\ 
t_3&=\frac{\Im z_1}{\Im\tau}, & &t_4=\Im z_2.
\end{align*}
By the isomorphism $\phi : X \to \C^\ast\times T_{\R}^2$ ; $(z_1, z_2)+\Lambda_{\tau,p,q} \to (e^{2\pi\sqrt{-1}(t_2+\sqrt{-1}t_4)}, e^{2\pi\sqrt{-1}t_1}, e^{2\pi\sqrt{-1}t_3})$, we obtain the commutative diagram
\begin{align*}
\xymatrix{\ar @{} [dr] 
X \ar[d]_{\pi} \ar[r]^\phi &  \C^\ast\times T_{\R}^2\ar[d]_{\tilde{\pi}} \\
\C/\langle 1,\tau \rangle \ar[r]^\sigma & T_{\R}^2 },
\end{align*}
where $\tilde{\pi}$ is the projection onto the $T_{\R}^2$-component and $\sigma: \C/\langle 1,\tau \rangle \to T_{\R}^2; z_1 \to (e^{2\pi\sqrt{-1}t_1}, e^{2\pi\sqrt{-1}t_3})$.

\end{subsection}

\begin{subsection}{Several sheaves on a toroidal group $X=\C^2/\Lambda_{\tau,p,q}$}\label{sheaves}
In this subsection, we define several sheaves utilized in the proof. 
We decompose the differential operator with respect to each variable as follows:
\begin{align*}
\del=\del_{z_1}+\del_{z_2},\quad \bar{\del}=\bar{\del}_{z_1}+\bar{\del}_{z_2}.
\end{align*}
We define the following sheaves on $X$ as the kernel of differential operators: 
\begin{align*}
\mathcal{F}&\coloneqq \Ker (\bar{\del}_{z_2}: \mathcal{A}\to\mathcal{A}^{0,1}), \\
\bar{\mathcal{F}}&\coloneqq  \Ker (\del_{z_2}: \mathcal{A}\to\mathcal{A}^{1,0}), \\
\mathcal{G}&\coloneqq \Ker (\del_{z_2}\bar{\del}_{z_2}: \mathcal{A}\to\mathcal{A}^{1,1})
\end{align*}
Subsequently, for each open set $U\subset X$, we define the presheaves
\begin{align*}
\mathcal{F}^{p,q}(U)\coloneqq \{ \sum_{I, J} f_{I,J} dz_I\wedge d\bar{z}_J &\mid f_{I,J}\in\mathcal{F}(U), I=\{ i_1\leq \dots\leq i_p \}\subset\{1,2\},\\
&J=\{ j_1\leq\dots\leq j_q \}\subset\{1\} \}, \\
\bar{\mathcal{F}}^{p,q}(U) \coloneqq \{ \sum_{I, J} f_{I,J} dz_I\wedge d\bar{z}_J &\mid f_{I,J}\in\bar{\mathcal{F}}(U), I=\{ i_1\leq \dots\leq i_p \}\subset\{1\},\\
&J=\{ j_1\leq\dots\leq j_q \}\subset\{1, 2\}\}, \\
\mathcal{G}^{1,0}(U) \coloneqq \{ f_1 dz_1+f_2dz_2 &\mid f_1\in\mathcal{G}(U), f_2\in\mathcal{F}(U) \},\\
\mathcal{G}^{0,1}(U) \coloneqq \{ f_1 d\bar{z}_1+f_2d\bar{z}_2 &\mid f_1\in\mathcal{G}(U), f_2\in\bar{\mathcal{F}}(U) \},\\
\mathcal{G}^{1,1}(U) \coloneqq \{ f_1 dz_1\wedge d\bar{z}_1 +f_2dz_2\wedge d\bar{z}_1 &+f_3 dz_1\wedge d\bar{z}_2 
\mid f_1\in\mathcal{G}(U), f_2\in\mathcal{F}(U), f_3\in\bar{\mathcal{F}}(U) \}.
\end{align*}
Here, we define $dz_I \coloneqq dz_{i_1}\wedge \cdots dz_{i_n}$ for multi-indices $I=\{i_1, \dots, i_n\}$.
Then, we denote the sheaves induced by each presheaf by $\mathcal{F}^{p,q}, \bar{\mathcal{F}}^{p,q}$ and $\mathcal{G}^{p,q}$.

Indeed, the presheaves defined in this manner become sheaves.
\begin{lem}
Presheaves $\mathcal{F}^{p,q}, \bar{\mathcal{F}}^{p,q}$ and $\mathcal{G}^{p,q}$ are sheaves on $X$. 
\end{lem}
\begin{proof}
First, we show that $\mathcal{F}^{p,q}$ is a sheaf. 
Take an open set $U\subset X$ and a sufficiently fine open covering $\{U_j\}$ of $U$. 
Let $f$ and $g$ be elements in $\mathcal{F}^{p,q}(U)$ such that $f|_{U_j}=g|_{U_j}$ on each $U_j$. 
Then, $f=g$ on $U$ by their continuity. 
Assume that $f_j=f_k$ on each $U_{jk}\coloneqq U_j\cap U_k$ for a family of sections $\{f_j, U_j\}$. 
Now, a local coordinate $z^j=(z^j_1, z^j_2)\in U_j\subset X$ is given by the standard coordinate of the universal covering $\C^2$ of $X$. 
Hence, every coordinate transformation between any $U_j$ and $U_k$ is translation. 
Let $\phi^{jk}: U_{jk}\to U_{jk}; z^k \mapsto z^j$ be a coordinate transformation on each $U_{jk}$. 
Then $dz^j_I\wedge d\bar{z}^j_J= dz^k_I\wedge d\bar{z}^k_J$ holds on each $U_{jk}$. 
We have that 
\begin{align*}
f_j=\sum_{I,J} f^j_{I,J}(z^j) dz^j_I\wedge d\bar{z}^j_J 
=\sum_{I,J} f^j_{I,J}(\phi^{jk}(z^k)) dz^k_I\wedge d\bar{z}^k_J
\end{align*}
holds on each $U_{jk}$. 
Therefore, the element $F\in\mathcal{A}^{p,q}(U)$ defined as $F|_{U_j}\coloneqq f_j$ on each $U_j$ is in $\mathcal{F}^{p,q}(U)$ because $f^j_{I,J}(\phi^{jk}(z^k))$ is $\delbar_{z_2}$-closed. 
From this, $\mathcal{F}^{p,q}$ is a sheaf on $X$. 
By a similar method, we can prove that $\bar{\mathcal{F}}^{p,q}$ and $\mathcal{G}^{p,q}$ are sheaves.

\end{proof}
\end{subsection}
\begin{subsection}{Local resolution lemma}\label{local}
In this subsection, we demonstrate the validity of the local resolution lemma, which is equivalent to a certain complex is partially exact. 
This complex is used in the study of Bott--Chern and Aeppli cohomologies. 

Let us consider the operator $\del +\delbar : \mathcal{A}^{(p,q)+1} \coloneqq \mathcal{A}^{p+1,q} \oplus \mathcal{A}^{p,q+1} \to \mathcal{A}^{p+1,q+1}$ defined by $(f,g) \mapsto \delbar f + \del g$. 
We then consider the following complex: 
\begin{align}
\cdots \overset{\del\delbar}{\to} \mathcal{A}^{p-1,q-1} \overset{d}{\to} \mathcal{A}^{(p-1,q-1)+1} \overset{\del +\delbar}{\to} \mathcal{A}^{p,q} \overset{\del\delbar}{\to} \mathcal{A}^{p+1,q+1} \overset{d}{\to} \cdots. 
\end{align}
This complex is partially exact. 
That is, the following lemma holds.

\begin{lem}[{\cite[Lemma $5.15$]{DGMS} (See also \cite[Lemma $4.1$]{Sch}.)}]\label{Sch_ex}
Let $U \subset \C^n$ be a sufficiently small ball. 
\begin{enumerate}[(i)]
\item Let $\theta$ be a $C^\infty$ $k$-form on $U$, and let $p_1$ and $p_2$ be positive integers. 
Assume that $k\geq1$ and the non-zero components of $\theta$ have bidegree $(p,q)$ satisfying $p_1 \leq p \leq p_2$. 
If $\theta$ is $d$-closed, then there exists a $(k-1)$-form $\alpha$ such that $\theta=d\alpha$ and its non-zero components have bidegree $(p,q)$ with $p_1\leq p \leq p_2-1$. 

\item Let $\theta$ be a $d$-closed $C^\infty$ $(p,q)$-form on $U$. 
Then, the following assertions hold:
\begin{enumerate}
\item If $p\geq1$ and $q\geq1$, $\theta\in \del\delbar\mathcal{A}^{p-1,q-1}(U)$.
\item If $p\geq1$ and $q=0$, $\theta\in \del\Omega^{p-1}(U)$.
\item If $p=0$ and $q\geq1$, $\theta\in \delbar\,\bar{\Omega}^{q-1}(U)$.
\item If $p=q=0$, $\theta$ is constant.
\end{enumerate}

\item Let $\theta$ be a $\del\delbar$-closed $C^\infty$ $(p,q)$-form on $U$. 
Then, $\theta$ can be expressed as the sum of a $\delbar$-closed form and a $\del$-closed form on $U$. 
In other words,  
\begin{enumerate}
\item If $p\geq1$ and $q\geq1$, $\theta\in \delbar \mathcal{A}^{p,q-1}(U)+\del \mathcal{A}^{p-1,q}(U)$.
\item If $p\geq1$ and $q=0$, $\theta\in \Omega^{p}(U)+\del\mathcal{A}^{p-1,0}(U)$.
\item If $p=0$ and $q\geq1$, $\theta\in \delbar\mathcal{A}^{0,q-1}(U)+\overline\Omega^{q}(U)$.
\item If $p=q=0$, $\theta\in \mathcal{O}(U)+\overline{\mathcal{O}}(U)$.
\end{enumerate}
\end{enumerate}
\end{lem}

\begin{proof}
$(i)$: 
Let $\theta$ be a $d$-closed form that satisfies the assumption. 
By Poincar\'{e}'s lemma, there exists a $(k-1)$-form $\beta$ such that $\theta=d\beta$. 
If $p_1=0$ and $p_2=k$, the proof is complete. 
Therefore, it suffices to show for $k\geq2$ and $p_1>0$. 
When $p_1\geq1$, the non-zero components of $\theta$ all have bidegree $(p,q)$ with $p\geq1$. 
Hence, $\delbar \beta^{0,k-1}=0$, where $\beta^{p,q}$ denotes the $(p,q)$-components of $\beta$. 
By Dolbeault's lemma, there exists a  $(0,k-2)$-form $\gamma$ such that $\beta^{0,k-1}=\delbar \gamma$. 
We set $\tilde{\beta}\coloneqq \beta- d\gamma$. 
Then $d\tilde{\beta}=d\beta=\theta$ and $\tilde{\beta}^{0,k-1}=0$ by the definition of $\tilde{\beta}$. 
By replacing $\beta$ with $\tilde{\beta}$, we can choose $\beta$ to have no $(0,k-1)$-component. 
This $\beta$ satisfies $\delbar \beta^{1,k-2}=0$. 
Reapplying Dolbeault's lemma to $\beta$, we can further choose $\beta$ to have no $(1, k-2)$-component. 
Finally, by iterating this argument, we obtain a form $\beta$ with no $(p,q)$-components for $p<p_1$. 

On the other hand, this $\beta$ also satisfies $\del\beta^{k-1,0}=0$ since $p_2<k$. 
Using the conjugate of Dolbeault's lemma, we obtain a $(k-2,0)$-form $\hat{\gamma}$ such that $\beta^{k-1,0}=\del\hat{\gamma}$. 
Then $d\hat{\beta}=d\beta=\theta$, where $\hat{\beta}\coloneqq \beta- d\hat{\gamma}$. 
By replacing $\beta$ with $\hat{\beta}$, we may assume that $\beta$ has no $(k-1,0)$-component. 
Repeating this argument as before, we obtain a form $\beta$ with no $(p_2,k-p_2-1)$-component.\\

$(ii)$: 
(d) is obvious. 
(b) and (c) are equivalent. 
Hence, it suffices to show (a) and (b). 
Assume that $p\geq1$. 
Set $\tilde{\theta}^{p-1,q+1}\coloneqq 0$ and $\tilde{\theta}^{p,q}\coloneqq \theta$. 
Applying the assertion $(1)$ to $\theta=\tilde{\theta}^{p-1,q+1} + \tilde{\theta}^{p,q}$, we obtain a $(p-1,q)$-form $\alpha$ such that $\theta=d\alpha$. 
This $\alpha$ satisfies $\theta=\del\alpha$ and $\delbar\alpha=0$. 
Moreover, if $q>1$, Dolbeault's lemma ensures the existence of a form $\gamma$ such that $\alpha=\delbar\gamma$. 

$(iii)$: 
Set $\theta^{p+1,q}\coloneqq \del\theta$. 
Since $\theta^{p+1,q}$ is $d$-closed, assertion $(2)$ guarantees the existence of a $(p,q)$-form $\alpha$ such that $\theta^{p+1,q}=\del\alpha$ and $\delbar\alpha=0$. 
Thus, the assertion is established by considering the representation $\theta=(\theta-\alpha)+\alpha$. 
\end{proof}

\end{subsection}


\begin{subsection}{Known cohomologies for two-dimensional theta toroidal groups}\label{Dol}
In this subsection, we recall the known cohomologies for two-dimensional theta toroidal groups $X= \C^2/\Lambda_{\tau,p,q}$. 

Let $X=\C^2/\Lambda_{\tau,p,q}$ be a toroidal group. 
A fine resolution of the sheaf $\Omega^p$ on $X$
\begin{align*}
0 \to \Omega^p \overset{i}{\hookrightarrow} \mathcal{A}^{p,0} \overset{\delbar}{\to} \mathcal{A}^{p,1} \overset{\delbar}{\to} \mathcal{A}^{p,2} \to 0
\end{align*}
induces the Dolbeault cohomology. 
We observe another resolution of $\Omega^p$.
\begin{lem}[{\cite[Proposition $3.4$]{Kazama}(See also \cite[p.$49$]{AK}.)}]\label{Kazama_sq}
The following sequence of sheaves on $X$ is exact:
\begin{align}\label{ex0_1}
0\to\Omega^p \overset{i}{\hookrightarrow} \mathcal{F}^{p,0} \overset{\delbar}{\to} \mathcal{F}^{p,1}  \to 0. 
\end{align}
\end{lem}
\begin{proof}
It suffices to show that the map $\delbar : \mathcal{F}^{p,0} \to \mathcal{F}^{p,1}$ is surjective. 
Let $U\subset X$ be a sufficiently small open ball. 
Note that elements of $\mathcal{F}(U)$ and $\mathcal{F}^{0,1}(U)$ are $\delbar_{z_2}$-closed. 
Then, this complex can be regarded as the complex $(\mathcal{F}^{p,q}, \delbar_{z_1})$. 
Subsequently, we verify that the map $\delbar_{z_1}: \mathcal{F}^{p,0}(U) \to \mathcal{F}^{p,1}(U)$ is surjective. 
By the definition of $\mathcal{F}^{p,1}(U)$, all the elements in $\mathcal{F}^{p,1}(U)$ are $\delbar_{z_1}$-closed. 
Therefore $\delbar_{z_1}$ is surjective by Dolbeault's Lemma. 
\end{proof}

Let $\mathscr{U}\coloneqq\{U_j\}$ be a finite open covering of the elliptic curve $\C/\langle 1,\tau \rangle$, where each $U_j$ is a coordinate open ball. 
Then $\mathscr{V}\coloneqq\{\pi^{-1}(U_j)\}$ is a finite open covering of $X$. 
\begin{lem}[{\cite[Proposition $2.4$]{Kazama}(See also \cite[Lemma $2.2.4$]{AK}}]\label{Fvani}
Let $X= \C^2/\Lambda_{\tau,p,q}$ be a two-dimensional toroidal group. 
Then, for any $k\geq1$, the $k$-the cohomology $H^k(\mathscr{V}, \mathcal{F}^{p,q})$ vanishes. 
\end{lem}
\begin{proof}
Refer to \cite[Lemma $2.2.4$]{AK}. 
It suffices to prove that the covering $\mathscr{V}$ is acyclic for the sheaf $\mathcal{F}^{p,q}$. 
Note that $\pi^{-1}(U_j) \cong U_j \times \C^\ast$, and 
\begin{align*}
\mathcal{F}^{p,q}(U_j\times\C^\ast) \cong \bigoplus_{p'+p''=p} \mathcal{A}^{p',q}(U_j) \otimes \Omega^{p''}(\C^\ast) \quad {\text{(See Remark \ref{tensor})}}. 
\end{align*}
By K\"{u}nneth's formula \cite{Kaup} and simple calculations, we get $H^k(U_j\times\C^\ast, \mathcal{F}^{p,q})=0$ for any $k\geq1$. 

Let $\{\rho_j\}$ be a partition of unity subordinate to $\mathscr{U}$. 
Then, each $\pi^\ast\rho_j$ is $\delbar_{z_2}$-closed. 
For a cocycle $\{(U_{j_0,\dots,j_k}, f_{j_0,\dots,j_k})\}\in \check{Z}^k(\mathscr{V}, \mathcal{F}^{p,q})$, a $\delta$-primitive solution is given by
\begin{align*}
\{ (U_{j_0,\dots,j_{p-1}}, \sum_j (\pi^\ast\rho_j) f_{j,j_0,\dots,j_{k-1}}) \}\in \check{C}^{k-1}(\mathscr{V}, \mathcal{F}^{p,q}).
\end{align*}
Therefore, $\mathscr{V}$ is acyclic for the sheaf $\mathcal{F}^{p,q}$.
\end{proof}

\begin{remark}\label{tensor}
Let $U_j \in \mathscr{U}=\{U_j\}_{j\in\Z_{>0}}$ be as in the proof of Lemma \ref{Fvani}. 
We take an increasing sequence $\{D_k\}$ (resp. $\{\tilde{D}_k\}$) of relatively compact subsets $D_k\Subset U_j$ (resp. $\tilde{D}_k\Subset \C^\ast$) such that $\cup D_k=U_j$ (resp. $\cup \tilde{D}_k=\C^\ast$). 
For the sheaf $\mathcal{F}$ on a toroidal group $X= \C^2/\Lambda_{\tau,p,q}$, we can obtain the following approximation in terms of Sobolev spaces: 
\begin{align*}
\mathcal{F} (U_j\times\C^\ast) = \lim_{\underset{k\in\Z_{>0}}{\longleftarrow}} \left( H^k(D_k \times \tilde{D}_k) \cap \Ker \delbar_{z_2} \right), 
\end{align*}
where the space $H^k(D_k\times \tilde{D}_k) (=W^{k,2}(D_k\times \tilde{D}_k))$ is a Sobolev space and, in particular, a separable Hilbert space. 
By means of the (topological) tensor product $\hat{\otimes}$ of Hilbert spaces using orthonormal bases, we have 
\begin{align*}
 H^k(D_k \times \tilde{D}_k) \cap \Ker \delbar_{z_2}  \cong H^k(D_k) \hat{\otimes} (H^k (\tilde{D}_k)\cap\Ker \delbar) . 
\end{align*}
Since $\mathcal{A}(U_j)= \lim_{\underset{k\in\Z_{>0}}{\longleftarrow}} \left( H^k(D_k) \right)$ and $\mathcal{O}(\C^\ast) =  \lim_{\underset{k\in\Z_{>0}}{\longleftarrow}} \left(H^k (\tilde{D}_k)\cap\Ker \delbar \right) $ are nuclear Fr\'{e}chet spaces, it follows that
\begin{align*}
 \lim_{\underset{k\in\Z_{>0}}{\longleftarrow}} \left(H^k(D_k) \hat{\otimes} (H^k (\tilde{D}_k)\cap\Ker \delbar) \right) \cong \mathcal{A}(U_j) \otimes \mathcal{O}(\C^\ast), 
\end{align*}
where $\otimes$ denotes the tensor product over $\C$ (see \cite{PG}).
Therefore, we conclude that  $\mathcal{F}(U_j\times \C^\ast )\cong \mathcal{A}(U_j) \otimes \mathcal{O}(\C^\ast)$. 
\qed
\end{remark}

From Lemma $\ref{Kazama_sq}$ and Lemma $\ref{Fvani}$, we obtain the following theorem. 
\begin{prop}[{\cite[Proposition $3.4$]{Kazama}}]\label{Kazama_Dol}
\begin{align*}
H^{p,q} (X, \mathcal{O}) \cong \frac{\Ker (\delbar: \mathcal{F}^{p,q}(X) \to \mathcal{F}^{p,q+1}(X))}{\Im (\delbar: \mathcal{F}^{p,q-1}(X) \to \mathcal{F}^{p,q}(X))}.
\end{align*}
\end{prop}

\begin{lem}[{\cite[Theorem $4.3$]{Kazama}(See also \cite[$2.2.5$ Proposition]{AK}.)}]\label{dol_repre}
Let $X= \C^2/\Lambda_{\tau,p,q}$ be a two-dimensional theta toroidal group. 
Then, every $w\in\Ker (\delbar: \mathcal{F}^{p,q}(X) \to \mathcal{F}^{p,q+1}(X))$ is $\delbar$-exact, up to a form on $X$ with constant coefficients. 
That is, there exists $\eta\in\mathcal{F}^{p,q-1}(X)$ such that 
\begin{align*}
w=\delbar\eta + \C\bigwedge^p\{dz_1,dz_2\} \wedge \bigwedge^q\{d\bar{z}_1\}
\end{align*}
\end{lem}
\begin{proof}
Refer to \cite[$2.2.5$ Proposition]{AK}. 
\end{proof}

By Proposition $\ref{Kazama_Dol}$ and Lemma $\ref{dol_repre}$, we can determine the Dolbeault cohomology of $X$. 
Let $h_{\delbar}^{p,q}$ denote the dimension of the Dolbeault cohomology $H^{p,q}(X, \mathcal{O})$.
\begin{cor}[{\cite[Theorem $4.3$]{Kazama}}]\label{dol-cohom}
The Dolbeault numbers of two-dimensional theta toroidal groups are:
\begin{align}
 \begin{array}{ccccc}
     &  & h^{0,0}_{\delbar}=1 &  &  \\
     & h^{1,0}_{\delbar}=2 &  & h^{0,1}_{\delbar}=1 &  \\
    h^{2,0}_{\delbar}=1 &  & h^{1,1}_{\delbar}=2 &  & h^{0,2}_{\delbar}=0 \\
     & h^{2,1}_{\delbar}=1 &  & h^{1,2}_{\delbar}=0 &  \\
     &  & \phantom{\;.} h^{2,2}_{\delbar}=0 \;. &  & 
 \end{array}
 \end{align}
\end{cor}
In fact, the Dolbeault cohomologies of a general toroidal group $C^n/ \Gamma$ are computed in \cite{Kazama}. 

\begin{remark}\label{del-cohom}
The conjugate cohomology of the Dolbeault cohomology
\begin{align*}
H^{p,q}_{\del}(X) \coloneqq \Ker \del \cap \mathcal{A}^{p,q}(X)/ \Im \del
\end{align*}
can be easily derived from Corollary $\ref{dol-cohom}$. 
The numbers $h_\del^{p,q}\coloneqq \dim H^{p,q}_\del (X)$ are:
\begin{align}
 \begin{array}{ccccc}
     &  & h^{0,0}_{\del}=1 &  &  \\
     & h^{1,0}_{\del}=1 &  & h^{0,1}_{\del}=2 &  \\
    h^{2,0}_{\del}=0 &  & h^{1,1}_{\del}=2 &  & h^{0,2}_{\del}=1 \\
     & h^{2,1}_{\del}=0 &  & h^{1,2}_{\del}=1 &  \\
     &  & \phantom{\;.} h^{2,2}_{\del}=0 \;. &  & 
 \end{array}
 \end{align}
 \end{remark}
 
Next, we analyze the de Rham cohomology of two-dimensional toroidal groups $X$. 
We apply Fact $\ref{dexact}$ to $X$. 
Consequently, the following assertion is straightforward. 
\begin{lem}[\cite{Ume}]\label{dexact_rmk}
Let $X$ be a two-dimensional toroidal group. 
For every $d$-closed $C^\infty$ $k$-form $\phi$ on $X$, there exist uniquely the $k$-form on $X$ with constant coefficients, and a $C^\infty$ $(k-1)$-form $\psi$ on $X$ such that
\begin{align*}
\phi=  \sum \C \wedge^k \{dz_1, dz_2, d\bar{z}_1 \} + d\psi
\end{align*}

\end{lem}

\begin{cor}[{\cite[Theorem $2.1$]{Ume}}]\label{betti}
The betti numbers $b_i$ of two-dimensional toroidal groups are:
\begin{align*}
b_0=1, \quad b_1=3,\quad b_2=3,\quad b_3=1,\quad b_4=0.
\end{align*}
According to \cite[Theorem $2.1$]{Ume}, the de Rham cohomology groups of a general toroidal group $\C^n/\Gamma$  are specified.
\end{cor}

\end{subsection}
\end{section}

\begin{section}{Proof of Theorem $\ref{BCcohom}$}\label{pf_11}
First, we identify $H^{0,0}_{BC}(X)$ and $H^{2,2}_{BC}(X)$. 
Since $d$-closed $C^\infty$ functions on toroidal groups are constant, $h_{BC}^{0,0}=1$. 
By Lemma $\ref{dexact_rmk}$, $C^\infty$ $4$-forms are $d$-exact. 
Furthermore, from Fact $\ref{ddbar}$, we find that $d$-exact forms are $\del\delbar$-exact. 
Thus, we have $h^{2,2}_{BC}=0$.

Next, we observe $H^{1,0}_{BC}(X)$. 
Note that
\begin{align*}
H^{1,0}_{BC}(X)=\Ker d\cap \mathcal{A}^{1,0}(X) \subset \Ker\delbar \cap \mathcal{A}^{1,0}(X)=H^{1,0}_{\delbar}(X). 
\end{align*}
By Corollary $\ref{dol-cohom}$, we have that $h_{BC}^{1,0}\leq 2$.
On the other hand, since $dz_1$ and $dz_2$ are in $H^{1,0}_{BC}(X)$, it is clear that $h_{BC}^{1,0}\geq 2$. 
Thus, $h_{BC}^{1,0}= 2$. 
In a similar manner, we can prove that its conjugate $h^{0,1}_{BC}=2$.

Following this, we focus on $H^{2,0}_{BC}(X)$. 
Since 
\begin{align*}
H^{2,0}_{BC}(X)=\Ker d\cap \mathcal{A}^{2,0}(X) \subset \Ker\delbar \cap \mathcal{A}^{2,0}(X)=H^{2,0}_{\delbar}(X), 
\end{align*}
Corollary $\ref{dol-cohom}$ implies that $h_{BC}^{2,0}\leq 1$. 
In contrast, $h^{2,0}_{BC}\geq 1$, because $dz_1\wedge dz_2 \in H^{2,0}_{BC}(X)$. 
Hence, $h^{2,0}_{BC}=1$. 
Then $h^{0,2}_{BC}=1$ for its conjugate $H^{0,2}_{BC}(X)$. 

Subsequently, we specify $H^{2,1}_{BC}(X)$. 
Take a $d$-closed $C^\infty$ $(2,1)$-form $\phi$ on $X$. 
By Lemma $\ref{dexact_rmk}$, a $3$-form $\phi$ is of the form
\begin{align}\label{BC3form}
\phi =C dz_1\wedge dz_2 \wedge d\bar{z}_1+ d\psi, 
\end{align}
where $\psi$ is a $C^\infty$ $2$-form and $C\in\C$ is the uniquely determined constant. 
Then $d\psi$ is also a $(2,1)$-form. 
Furthermore, there exists a $(1,0)$-form $\tilde{\psi}$ such that $d\psi=\del\delbar\psi$ by Fact $\ref{ddbar}$. 
Thus, $h_{BC}^{2,1}\leq 1$. 
Now, $0\neq\left[ dz_1\wedge dz_2 \wedge d\bar{z}_1\right]\in H^{2,1}_{BC}(X)$. 
This is an immediate consequence of Fact $\ref{ddbar}$ and the uniqueness stated in Lemma $\ref{dexact_rmk}$. 
Indeed, it follows from the proof of Lemma $\ref{dexact_rmk}$ that $dz_1\wedge dz_2 \wedge d\bar{z}_1$ cannot be $d$-exact. 
Therefore $h_{BC}^{2,1}=1$. 
We can also calculate its conjugate $H^{1,2}_{BC}(X)$ as above.
In this case, $\left[ dz_1\wedge d\bar{z}_1 \wedge d\bar{z}_2\right]$ alone forms a basis for $H^{1,2}_{BC}(X)$. 
Thus, $h^{1,2}_{BC}=1$. 

Finally, we observe $H^{1,1}_{BC}(X)$. 
Take a $d$-closed $C^\infty$ $(1,1)$-form $\xi$ on $X$. 
By Lemma $\ref{dexact_rmk}$, there exist a $1$-form $\tilde{\xi}$ and the constants $C_i$ such that 
\begin{align*}
\xi = C_1 dz_1\wedge d\bar{z}_1 + C_2 dz_2\wedge d\bar{z}_1+ C_3 dz_1\wedge dz_2 + d\tilde{\xi}.
\end{align*}
Since $dz_2= d(2\i\Im z_2)+ d\bar{z}_2$, we have
\begin{align*}
\xi = C_1 dz_1\wedge d\bar{z}_1 + C_2 dz_2\wedge d\bar{z}_1+ C_3 dz_1\wedge d\bar{z}_2+ d(\tilde{\xi}-2C_3\i\Im z_2 dz_1). 
\end{align*}
Then $d(\tilde{\xi}-2C_3\i\Im z_2 dz_1)$ is $\del\delbar$-exact by Fact $\ref{ddbar}$. 
Hence, $h^{1,1}_{BC}\leq 3$. 
On the other hand, $\left[ dz_1\wedge d\bar{z}_1 \right], \left[ dz_2\wedge d\bar{z}_1 \right]$ and $\left[ dz_1\wedge d\bar{z}_2 \right] \in H^{1,1}_{BC}(X)$ are linearly independent. 
In fact, we can verify the independence by using Fact $\ref{ddbar}$ and the uniqueness of representation in Lemma $\ref{dexact_rmk}$.
Therefore, $h^{1,1}_{BC}=3$. 

\begin{remark}\label{BC1_1}
Indeed, we can obtain a vector space isomorphic to $H^{1,1}_{BC}(X)$ as shown below: 
First, we consider the exact sequence
\begin{align*}
0 \to \mathcal{H} \to \mathcal{A} \overset{\del\delbar}{\to} \mathcal{A}^{1,1} \overset{d}{\to} \mathcal{A}^3 \overset{d}{\to} \cdots.
\end{align*}
Since this sequence is a fine resolution of the sheaf $\mathcal{H}$ of germs of harmonic functions on $X$, we obtain the isomorphism
\begin{align*}
H^1(X,\mathcal{H})\cong \frac{\Ker (d: \mathcal{A}^{1,1}(X) \to \mathcal{A}^2(X))}{\Im (\del\delbar: \mathcal{A}(X) \to \mathcal{A}^{1,1}(X))} = H^{1,1}_{BC}(X).
\end{align*}
\end{remark}

\end{section}

\begin{section}{Aeppli cohomology of two-dimensional non-compact toroidal groups}\label{sec_Aeppli}
In this section, we consider theta toroidal groups $X =\C^2/\Lambda_{\tau,p,q}$. 
However, we remark that the isomorphisms between the cohomology and a vector space, which we prove in this section, hold for a general toroidal group.
We observe that the determination of Aeppli cohomology groups of type $(0,1)$ and of type $(1,1)$ can be reduced to a problem of convergence of power series. 

\begin{subsection}{The special sheaves on toroidal groups}\label{acyclic}

Let $\mathscr{U}\coloneqq\{U_j\}$ be a sufficiently fine finite Stein covering of the elliptic curve $\C/\langle 1,\tau \rangle$. 
Then $\mathscr{V}\coloneqq\{\pi^{-1}(U_j)\}$ is finite Stein covering of $X$. 
We prove that the covering $\mathscr{V}$ is acyclic for the sheaf $\mathcal{G}^{1,0}\oplus \mathcal{G}^{0,1}$. 
\begin{prop}\label{G_vani}
For any $p\geq 1$, 
\begin{align*}
H^p(\mathscr{V}, \mathcal{G}^{1,0}\oplus \mathcal{G}^{0,1})=0.
\end{align*}
\end{prop}
\begin{proof}
It suffices to show that $H^p(\mathscr{V},\mathcal{G}^{1,0})=0$. 
By the local trivialization of $\mathscr{V}=\{\pi^{-1}(U_j) \}$, 
\begin{align*}
\pi^{-1}(U_j)\cong U_j\times \C^\ast.
\end{align*}
By an argument similar to Remark \ref{tensor}, 
\begin{align*}
\mathcal{G}^{1,0}(U_j\times \C^\ast) \cong ( \mathcal{A}^{1,0}(U_j)\otimes \mathcal{H}^{0,0}(\C^\ast))\oplus(\mathcal{A}^{0,0}(U_j)\otimes \Omega^1(\C^\ast)).
\end{align*}
Using K\"{u}nneth's formula \cite{Kaup}, we have
\begin{align*}
H^p(U_j\times \C^\ast, \mathcal{G}^{1,0}) = \bigoplus_{s+t=p} \left( H^s(U_j, \mathcal{A}^{1,0}) \otimes H^t(\C^\ast, \mathcal{H}^{0,0}) \right) \oplus \left( H^s(U_j, \mathcal{A}^{0,0}) \otimes H^t(\C^\ast, \Omega^1) \right)
\end{align*}
for any $p\geq1$. 
Then $H^s(U_j,\mathcal{A}^{k,0})=0$ and $H^s(\C^\ast, \Omega^1)=0$ for any $s\geq 1$. 
Hence, we are reduced to showing that $H^t(\C^\ast, \mathcal{H}^{0,0})=0$ for any $t\geq1$. 

From the short exact sequence $0\to \C\to \mathcal{O}\oplus\bar{\mathcal{O}}\to \mathcal{H}\to 0$ (see Remark $\ref{harm_ex}$), we obtain the long exact sequence
\begin{align*}
\cdots\to H^k(\C^\ast, \C)\to H^k(\C^\ast, \mathcal{O}\oplus\bar{\mathcal{O}})\to H^k(\C^\ast, \mathcal{H})\to H^{k+1}(\C^\ast, \C)\to \cdots.
\end{align*}
Note that $H^k(\C^\ast, \mathcal{O}\oplus\bar{\mathcal{O}})=0$ for any $k\geq1$, and $H^k(\C^\ast,\C)=0$ for any $k\geq2$. 
From the long exact sequence, we have that $H^k(\C^\ast, \mathcal{H})=0$ for any $k\geq1$. 
It follows that $H^p (U_j\times\C^\ast, \mathcal{G}^{1,0})=0$ for any $p\geq1$. 

Finally, we prove that $H^p(\mathscr{V}, \mathcal{G}^{1,0})=0$ for $p\geq1$. 
Let  $w\coloneqq \{(U_{j_0,\dots, j_p}, f_{j_0,\dots, j_p})\}\in \check{Z}^p(\mathscr{V}, \mathcal{G}^{1,0})$. 
Let $\{\rho_j\}$ be a partition of unity subordinate to the covering $\mathscr{U}=\{U_j\}$ of $\C/\langle 1,\tau \rangle$. 
Then $\{\pi^\ast \rho_j\}$ is a partition of unity subordinate to the covering $\mathscr{V}$ of $X$. 
Now, an element
\begin{align*}
\{(U_{j_0,\cdots, j_{p-1}}, \sum_{m}(\pi^\ast \rho_m) f_{m, j_0, \dots, j_{p-1}})\}\in \check{C}^{p-1}(\mathscr{V}, \mathcal{G}^{1,0})
\end{align*}
is a $\delta$-primitive solution for $w$. 
Therefore $H^p(\mathscr{V}, \mathcal{G}^{1,0})=0$ for $p\geq1$. 

\end{proof} 

The following statement follows immediately from the argument in the proof of Proposition $\ref{G_vani}$.
\begin{cor}\label{G_cor}
For any $p\geq1$, 
\begin{align*}
H^p(\mathscr{V}, \mathcal{G})=0.
\end{align*}
\end{cor}

\begin{remark}\label{harm_ex}
The sequence $0\to \C\to \mathcal{O}\oplus\bar{\mathcal{O}}\to \mathcal{H}\to 0$ is exact. 
Here, we consider the map $\C\to \mathcal{O}\oplus\bar{\mathcal{O}}; c\mapsto (c,-c)$, and $\mathcal{O}\oplus\bar{\mathcal{O}}\to \mathcal{H}; (f,g)\mapsto f+g$. 
In fact, the exactness follows from Lemma $\ref{Sch_ex}$. 

\end{remark}
\end{subsection}
\begin{subsection}{Aeppli cohomology of type $(0,0)$}\label{sec_Aeppli00}

In order to show an isomorphism between Aeppli cohomology of type $(0,0)$ and a certain vector space, we will use the following exact sequence of sheaves. 
\begin{lem}\label{fine00}
The sequence 
\begin{align*}
0\to \C\overset{\nu_0}{\to} \mathcal{O}\oplus\bar{\mathcal{O}}\overset{\nu_1}{\to} \mathcal{A}\overset{\del\delbar}{\to} \mathcal{A}^{1,1}\overset{d}{\to} \cdots
\end{align*}
is exact, where $\nu_0: c\mapsto (c,-c)$ and $\nu_1: (f,g)\mapsto f+g$.
\end{lem}
\begin{proof}
This follows immediately from the assertions in Lemma $\ref{Sch_ex}$ and Remark $\ref{harm_ex}$.
\end{proof}
From the sequence in Lemma $\ref{fine00}$, we obtain $H^{0,0}_A (X)\cong H^1(\mathscr{V}, \Ker \nu_1)$ for a Stein covering $\mathscr{V}$. 
Next, we consider the following exact sequence which decomposes $\Ker\nu_1$. 
\begin{lem}\label{ex00}
The sequence
\begin{align*}
0\to \Ker\nu_1 \overset{i}{\hookrightarrow} \mathcal{O}\oplus\bar{\mathcal{O}}\overset{\tilde{\nu_1}}{\to} \mathcal{G}\overset{\del\delbar}{\to} \Im\del\delbar \to 0
\end{align*}
is exact, where $\tilde{\nu_1}: (f,g)\mapsto f+g$
\end{lem}
\begin{proof}
Let $U\subset X$ be a sufficiently small open ball. 
Clearly, $\Ker\nu_1=\Ker\tilde{\nu_1}$. 
By the definition of $\tilde{\nu_1}$, $\Im\tilde{\nu_1}\subset\mathcal{G}(U)$. 
Let $w\in\Ker\del\delbar\cap\mathcal{G}(U)$. 
Since $w\in\mathcal{H}(U)$, there exist $f\in\mathcal{O}(U)$ and $g\in\bar{\mathcal{O}}(U)$ such that $w=f+g$. 
Hence, $\Im\tilde{\nu_1}=\mathcal{G}(U)$.
\end{proof}

\begin{prop}\label{cong00}
\begin{align*}
H^{0,0}_A(X)\cong \frac{\Ker\del\delbar\cap\mathcal{G}(X)}{\mathcal{O}(X)+\bar{\mathcal{O}}(X)}.
\end{align*}
\end{prop}
\begin{proof}
Note that $H^{0,0}_A (X)\cong H^1(\mathscr{V}, \Ker \nu_1)$ for a Stein covering $\mathscr{V}$. 
From Lemma $\ref{ex00}$, we obtain the following short exact sequence
\begin{align*}
0&\to \Ker\nu_1 \to \mathcal{O}\oplus\bar{\mathcal{O}} \to \Ker\del\delbar \to 0. 
\end{align*}
Using the long exact sequence induced by this short exact sequence, we obtain the following exact sequence of quotient spaces: 
\begin{align*}
0\to \Im(\nu_1: H^0(\mathscr{V}, \mathcal{O}\oplus\bar{\mathcal{O}}) \to H^1(\mathscr{V}, &\Im\nu_1))\to H^0(\mathscr{V}, \Ker\del\delbar)\\ 
&\to H^1(\mathscr{V}, \Ker\nu_1) \to H^1(\mathscr{V}, \mathcal{O}\oplus\bar{\mathcal{O}})=0.
\end{align*}
Thus, we obtain $H^1(\mathscr{V}, \Ker\nu_1)\cong \Ker\del\delbar\cap\mathcal{G}(X)/\mathcal{O}(X)+\bar{\mathcal{O}}(X)$.
\end{proof}

Subsequently, using Proposition $\ref{cong00}$, we identify Aeppli cohomology of type $(0,0)$. 
Take $w\in\Ker\del\delbar\cap \mathcal{G}(X)$. 
From the general theory of toroidal groups, we obtain the following Fourier series expansion for $w$: 
\begin{align}\label{Fourier}
w(z_1,z_2)= \sum_{\sigma=(\sigma_1, \sigma_2, \sigma_3)\in\Z^3} a^\sigma (\Im z_2) \exp\langle \sigma, t'\rangle, 
\end{align}
where the coefficients $a^\sigma$ are certain $C^\infty$ functions of $t_4=\Im z_2$, $t'=(t_1, t_2, t_3)$, and $\exp{\langle \sigma, t' \rangle}\coloneqq \exp{(2\pi\sqrt{-1}(\sigma_1 t_1+\sigma_2 t_2+ \sigma_3 t_3))}$. 
Here, each $t_i$ is defined as in \S $\ref{2-dim}$. 
 
We set $\textstyle A^\sigma\coloneqq \frac{\pi\sqrt{-1}}{\Im\tau} (\sigma_1(\Im\tau-\sqrt{-1}\Re\tau) +\sigma_2( -p(\Im\tau-\sqrt{-1}\Re\tau)-q\sqrt{-1}) +\sqrt{-1} \sigma_3)$ and $B^\sigma\coloneqq \i\pi\sigma_2$. 
We denote by $\bar{A^\sigma}$ the conjugate of $A^\sigma$. 
Then, since $\del\delbar w=0$, we find that
\begin{align*}
0=\del_{z_1}\delbar_{z_1}w= \sum_{\sigma\in\Z^3} \left(-|A^\sigma|^2\right)a^\sigma \exp\langle \sigma, t'\rangle,
\end{align*}
Note that $\sigma=0$ is equivalent to $A^\sigma=0$. 
It follows that $a^\sigma=0$ for any $\sigma\neq0$, that is, $w=a^0(t_4)$.
Furthermore, since $w\in \mathcal{G}(X)$, we have $\del_{z_2}\delbar_{z_2}w=0$. 
By simple calculations, $d^2 a^0(t_4)/d (t_4)^2=0$. 
From the above discussion, we have shown that $w=C_2 t_4+C_1$ for certain constants $C_1$ and $C_2$. 
Therefore, we conclude that $h^{0,0}_A=1$ from Proposition $\ref{cong00}$.

\end{subsection}


\begin{subsection}{Aeppli cohomology of type $(0,1)$}\label{sec_Aeppli01}
In this subsection, we also observe a vector space isomorphic to Aeppli cohomology of type $(0,1)$ using a similar method as in \S $\ref{sec_Aeppli00}$. 
\begin{lem}\label{fine01}
The sequence
\begin{align}
0\to \C \overset{\iota_0}{\to} \mathcal{O}\oplus\bar{\mathcal{O}} \overset{\iota_1}{\to} \mathcal{A}\oplus\bar{\Omega}^1 \overset{\iota_2}{\to} \mathcal{A}^{0,1} \overset{\del\delbar}{\to} \mathcal{A}^{1,2}\to 0 
\end{align}
is exact, where $\iota_0: c\mapsto (c,-c), \iota_1: (f,g)\mapsto (f+g,-\delbar g)$, and $\iota_2: (s,t)\mapsto \delbar s+t$.
\end{lem}
\begin{proof}
Let $U\subset X$ be a sufficiently small open ball. 
Clearly, $\iota_0$ is injective.
By Lemma $\ref{Sch_ex}$, we see that $\Im\iota_2=\Ker\del\delbar$ and $\del\delbar: \mathcal{A}^{0,1}(U)\to \mathcal{A}^{1,2}$ is surjective. 

It is clear that $\Im\iota_0\subset\Ker\iota_1$. 
Let $(f,g)\in\Ker\iota_1$. 
Then $g$ is holomorphic and anti-holomorphic on $U$. 
Hence, $g=-f$ is constant. 
We find that $\iota_0 (f)=(f,g)$, that is, $\Im\iota_0=\Ker\iota_1$. 

Finally, we prove $\Im\iota_1=\Ker\iota_2$. 
A straightforward calculation shows that $\Im\iota_1\subset\Ker\iota_2$. 
Take $(s,t)\in\Ker\iota_2$. 
Since $\delbar s+t=0$ and $t$ is $\del$-closed, we obtain $\del\delbar s=0$. 
By Lemma $\ref{Sch_ex}$, there exist $h_1\in \mathcal{O}(U)$ and $h_2\in\bar{\mathcal{O}}(U)$ such that $s=h_1+h_2$. 
Then, 
\begin{align*}
\iota_1(h_1,h_2)=(h_1+h_2,-\delbar h_2)= (s, -\delbar s)=(s,t).
\end{align*}
\end{proof}

Note that $\mathcal{O}$ and $\bar{\Omega}^q$ are coherent, and $\mathscr{V}$ is a Stein covering. 
It follows from Lemma $\ref{fine01}$ that 
\begin{align*}
H^{0,1}_A(X)&= \frac{\Ker\del\delbar\cap\mathcal{A}^{0,1}(X)}{\Im(\iota_2: \mathcal{A}(X)\oplus\bar{\Omega}^1(X)\to \mathcal{A}^{0,1}(X))}\\
&\cong H^1(\mathscr{V}, \Ker\iota_2)\\
&\cong H^2(\mathscr{V},\Ker\iota_1).
\end{align*}

Next, we consider the following exact sequence which decomposes $\Ker\iota_1$. 
\begin{lem}\label{ex01}
The sheaf sequence
\begin{align*}
0\to \Ker\iota_1 \hookrightarrow \mathcal{O}\oplus\bar{\mathcal{O}}\overset{\tilde{\iota}_1}{\to} \mathcal{G}\oplus\bar{\Omega}^1\overset{\tilde{\iota}_2}{\to} \mathcal{G}^{0,1}\overset{\del\delbar}{\to} \Im \del\delbar \to 0
\end{align*}
is exact, where $\tilde{\iota}_1: (f,g)\mapsto (f+g,-\delbar g)$ and $\tilde{\iota}_2: (s,t)\mapsto (\delbar s+t)$.
\end{lem}
\begin{proof}
Let $U\subset X$ be a sufficiently small open ball. 
Clearly, $\Im\tilde{\iota}_1\subset\Ker\tilde{\iota}_2$. 
Conversely, we take $(s,t)\in\Ker\tilde{\iota}_2$. 
Since $t$ is $\del$-closed, we have $\del\delbar s=0$. 
By Lemma $\ref{Sch_ex}$, there exist $h_1\in \mathcal{O}(U)$ and $h_2\in\bar{\mathcal{O}}(U)$ such that $s=h_1+h_2$. 
We then have $\tilde{\iota}_1(h_1,h_2)=(s,t)$, that is, $\Im\tilde{\iota}_1\supset\Ker\tilde{\iota}_2$. 

It is clear that $\Im\tilde{\iota}_2\subset\Ker\del\delbar$. 
Take $w\in\Ker\del\delbar$. 
Using Lemma $\ref{Sch_ex}$, we can write $w$ as $\delbar f+g$ for some $f\in\mathcal{A}(U)$ and $g\in\bar{\Omega}^1(U)$. 
In general, $w\in\mathcal{G}^{0,1}(U)$ can be written as $w=F_1d\bar{z}_1+F_2d\bar{z}_2$ for certain functions $F_1\in\mathcal{G}(U)$ and $F_2\in\bar{\mathcal{F}}(U)$. 
Then, 
\begin{align*}
\del\delbar f=\del w= \frac{\del F_1}{\del z_1}dz_1\wedge d\bar{z}_1+ \frac{\del F_1}{\del z_2}dz_2\wedge d\bar{z}_1+ \frac{\del F_2}{\del z_1}dz_1\wedge d\bar{z}_2. 
\end{align*}
Hence, we have $\del_{z_2}\delbar_{z_2} f=0$. 
Therefore, we find that $\Im\tilde{\iota}_2\supset\Ker\del\delbar$.
\end{proof}

Note that $\mathscr{V}$ is a Stein covering and $\bar{\Omega}^1$ is coherent. 
Using Corollary $\ref{G_cor}$, we obtain $H^2(\mathscr{V}, \Ker\iota_1)\cong \Ker\del\delbar\cap\mathcal{G}^{0,1}(X)/(\delbar\mathcal{G}(X)+\bar{\Omega}^1(X))$ in a similar manner to the proof of Proposition $\ref{cong00}$. 
This completes the proof of the following Proposition.
\begin{prop}\label{cong01}
\begin{align*}
H^{0,1}_A(X)\cong \frac{\Ker\del\delbar\cap\mathcal{G}^{0,1}(X)}{\delbar\mathcal{G}(X)+\bar{\Omega}^1(X)}.
\end{align*}
\end{prop}

\begin{remark}\label{cong10}
We can also investigate $H^{1,0}_A(X)$ using a similar argument. 
By considering the complex conjugate for Lemma $\ref{fine01}$ and Lemma $\ref{ex01}$, we find that 
\begin{align*}
H^{1,0}_A(X)\cong \frac{\Ker\del\delbar\cap\mathcal{G}^{1,0}(X)}{\del\mathcal{G}(X)+\Omega^1(X)}.
\end{align*} 
\end{remark}

In the following, using Proposition $\ref{cong01}$, we observe Aeppli cohomology of type $(0,1)$. 
Let $w=F_1d\bar{z_1}+F_2d\bar{z}_2\in \mathcal{G}^{0,1}(X)$, where $F_1\in\mathcal{G}(X)$ and $F_2\in\bar{\mathcal{F}}(X)$.
From the Fourier series expansion as $(\ref{Fourier})$, we have
\begin{align*}
F_i=\sum_{\sigma\in\Z^3}a_i^\sigma(\Im z_2)\exp\langle \sigma,t' \rangle.
\end{align*}
As $F_1\in\mathcal{G}(X)$, $\del_{z_2}\delbar_{z_2} (a_1^\sigma \exp{\langle \sigma,t' \rangle})=0$ for each $\sigma \in \Z^3$. 
We find that
\begin{align*}
0&=\del_{z_2}\delbar_{z_2} (a_1^\sigma \exp{\langle \sigma,t' \rangle})\\
&=\left(\frac{\del^2 a_1^\sigma}{\del \bar{z}_2 \del z_2} + B^\sigma \frac{\del a_1^\sigma}{\del \bar{z}_2} +B^\sigma \frac{\del a_1^\sigma}{\del z_2} +(B^\sigma)^2 a_1^\sigma \right)\exp{\langle \sigma,t' \rangle}d\bar{z}_2\wedge dz_2\\
&=\left( \frac{1}{4} \cdot \frac{d^2 a_1^\sigma}{dt_4^2} + (B^\sigma)^2 a_1^\sigma \right)d\bar{z}_2\wedge dz_2 
\end{align*}
for each $\sigma\in\Z^3$. 
The general solution to this differential equation can be written as follows in terms of constants $C_1^\sigma$ and $C_2^\sigma$:
\begin{align}\label{aeppli01_a1}
a_1^\sigma=
\left\{
\begin{array}{ll}
C_1^\sigma\exp{(2\pi \sigma_2 t_4)}+ C_2^\sigma \exp{(-2\pi \sigma_2t_4)} & (\sigma_2\neq 0) \\
C_2^\sigma t_4+ C_1^\sigma & (\sigma_2=0)
\end{array}.
\right.
\end{align}
As $F_2\in\bar{\mathcal{F}}(X)$, we see that 
\begin{align*}
0=\del_{z_2}(a^\sigma_2\exp{\langle \sigma,t' \rangle}) =\left( \frac{\del a^\sigma_2}{\del z_2}+ a^\sigma_2 B^\sigma\exp{\langle \sigma,t' \rangle}\right) 
\end{align*}
for each $\sigma\in\Z^3$. 
Hence, we obtain the general solution
\begin{align*}
a^\sigma_2=
\left\{
\begin{array}{ll}
\tilde{C}_1^\sigma\exp(2\pi \sigma_2 t_4) & (\sigma_2\neq 0) \\
\tilde{C}_1^\sigma & (\sigma_2=0)
\end{array}.
\right.
\end{align*}
Since $\del\delbar w=0$, we have 
\begin{align}
A^\sigma a^\sigma_2=\frac{\del a^\sigma_1}{\del \bar{z}_2}+a^\sigma_1 B^\sigma
\end{align}
for any $\sigma\neq 0$. 

Assume that there exist $\psi\in\mathcal{G}(X)$ and $\eta\in\bar{\Omega}^1(X)$ such that $w=\delbar\psi+\eta$. 
Here, 
\begin{align*}
\psi&=\sum_{\sigma\in\Z^3} b^\sigma_1\exp{\langle \sigma,t' \rangle},\\
\eta&=\sum_{\sigma\in\Z^3} b^\sigma_2\exp{\langle \sigma,t' \rangle}d\bar{z}_1+\sum_{\sigma\in\Z^3} b^\sigma_3\exp{\langle \sigma,t' \rangle}d\bar{z}_2
\end{align*}
from the Fourier series expansion. 
Then, the comparison of coefficients in $w=\delbar \psi+\eta$ yields the relations
\begin{align}\label{relation01}
\left\{
\begin{array}{ll}
a^\sigma_1=A^\sigma b^\sigma_1+b^\sigma_2, \\
a^\sigma_2=\frac{\del b^\sigma_1}{\del \bar{z}_2}+ B^\sigma b^\sigma_1 + b^\sigma_3
\end{array}
\right.
\end{align}
for any $\sigma\in\Z^3$. 

In order to satisfy the relation $(\ref{relation01})$, we set
\begin{align*}
b^\sigma_1 \coloneqq 
\left\{
\begin{array}{ll}
\frac{a^\sigma_1}{A^\sigma} & (\sigma\neq 0)\\
0 & (\sigma=0)
\end{array}
\right. ,
b^\sigma_2 \coloneqq
\left\{
\begin{array}{ll}
0 & (\sigma\neq 0)\\
C^\sigma_1 & (\sigma=0)
\end{array}
\right. ,
b^\sigma_3 \coloneqq 
\left\{
\begin{array}{ll}
0 & (\sigma\neq 0)\\
a^\sigma_2 & (\sigma=0)
\end{array}
\right. .
\end{align*}
Then, using $\textstyle \tilde{\psi}\coloneqq \sum_{\sigma\in\Z^3} b^\sigma_1\exp{\langle \sigma,t' \rangle}$ and $ \textstyle \tilde{\eta}\coloneqq\sum_{\sigma\in\Z^3} b^\sigma_2\exp{\langle \sigma,t' \rangle}d\bar{z}_1+\sum_{\sigma\in\Z^3} b^\sigma_3\exp{\langle \sigma,t' \rangle}d\bar{z}_2$, we obtain a formal solution
\begin{align*}
w= \delbar \tilde{\psi}+ \tilde{\eta} +C^0_2 t_4 d\bar{z}_1.
\end{align*}
We remark that $\tilde{\psi}\in\mathcal{G}(X)$ and $\tilde{\eta}\in\bar{\Omega}^1(X)$ if these are convergent. 
Moreover, $C^0_2$ is uniquely determined. 

\begin{remark}
From Remark $\ref{cong10}$, we can make a similar argument for $H^{1,0}_A(X)$.
In other words, we obtain a formal solution $w= \del  \tilde{\psi}+ \tilde{\eta} +C^0_2 t_4 dz_1$ for a given $w\in\Ker\del\delbar\cap\mathcal{G}^{1,0}(X)$. 
\end{remark}

\begin{obs}
We try to prove the convergence of $\tilde{\psi}$ and $\tilde{\eta}$. 
By the definition of $\tilde{\eta}$, $\tilde{\eta}$ is a constant form. 
Now, we recall that there exist $C>0$ and $0<\delta<1$ such that 
\begin{align*}
C\delta^n\leq \dist((np,nq),\Z^2)
\end{align*}
for any $n\in\Z_{>0}$, since $X$ is a theta toroidal group. 
It follows that there exists $D>0$ such that $\dist((\sigma_2 p,\sigma_2 q),\Z^2)\leq \dist((\sigma_2 p,\sigma_2 q),(\sigma_1, \sigma_3)) \leq D |A^\sigma|$ for any $\sigma\in\Z^3$ with $\sigma_2\neq0$. 

Based on the relations $(\ref{aeppli01_a1})$, we assume that $\textstyle \sum_{\sigma_2\neq0} |C^\sigma_1|\cdot |(e^{2\pi t_4})^{\sigma_2}|$ and $\textstyle \sum_{\sigma_2\neq0} |C^\sigma_2|\cdot |(e^{2\pi t_4})^{-\sigma_2}|$ converge uniformly on compact sets. 
That is, we find that $\textstyle \sum_{\sigma_2\neq0} |C^\sigma_1| k^{\sigma_2}$ and $\textstyle \sum_{\sigma_2\neq0} |C^\sigma_2| k^{\sigma_2}$ converge uniformly for any $k>0$.
Hence, we see that
\begin{align*}
\sum_{\sigma_2\neq0} \frac{|a^\sigma_1|}{|A^\sigma|} \leq D\sum_{\sigma_2\neq0} |a^\sigma_1| \delta^{|\sigma_2|} 
\leq \frac{C}{D} \left( \sum_{\sigma_2\neq0} |C^\sigma_1| \delta^{|\sigma_2|} |e^{2\pi t_4}|^{\sigma_2}+ \sum_{\sigma_2\neq0} |C^\sigma_2| \delta^{|\sigma_2|} |e^{2\pi t_4}|^{-\sigma_2}\right)
\end{align*}
is convergent. 
Furthermore, $\textstyle \sum_{\sigma_2=0} \frac{|a^\sigma_1|}{|A^\sigma|} \leq \sum_{\sigma_2=0} |a^\sigma_1|$ since $|A^\sigma|\geq1$ for $\sigma_2=0$. 
Thus, the series $\textstyle \sum_{\sigma_2=0} \frac{|a^\sigma_1|}{|A^\sigma|}$ is convergent. 
Therefore, we find that $\tilde{\psi}$ is convergent. 

\end{obs}

\end{subsection}


\begin{subsection}{Aeppli cohomology of type $(1,1)$}\label{sec_Aeppli11}
For the identification of $H^{1,1}_A(X)$, we next observe the following sequence. 

\begin{lem}\label{ex1_1}
The sheaf sequence
\begin{align}\label{sq1_1}
0\to \C \overset{\eps_0}{\to} \mathcal{O}\oplus\bar{\mathcal{O}} \overset{\eps_1}{\to} \Omega^1\oplus \bar{\Omega}^1 \oplus \mathcal{A} \overset{\eps_2}{\to} \mathcal{A}^{1,0}\oplus\mathcal{A}^{0,1} \overset{\delbar+\del}{\to} \mathcal{A}^{1,1} \overset{\del\delbar}{\to} \mathcal{A}^{2,2} \to 0
\end{align}
is exact, where $\eps_0; c \mapsto (c, -c)$, $\eps_1; (f,g)\mapsto (\del f, \delbar g, f+g) $ and $\eps_2 ; (h_1,h_2,\eta)\mapsto (h_1-\del \eta, h_2-\delbar\eta)$. 
\end{lem}
\begin{proof}
It is clear that $\eps_0$ is injective and $\Im \eps_0= \Ker \eps_1$.  
In addition, we have already proven the exactness at $\mathcal{A}^{1,1}$ and $\mathcal{A}^{2,2}$ in Lemma $\ref{Sch_ex}$. 
It suffices to prove that $\Im \eps_1=\Ker \eps_2$ and $\Im\eps_2=\Ker \delbar+\del$. 

Let $U\subset X$ be a sufficiently small open ball. 
It is clear that $\Im\eps_1\subset \Ker\eps_2$. 
Take $(h_1, h_2, \eta)\in \Ker\eps_2$. 
Then, we have $\del\delbar\eta=0$. 
By Lemma $\ref{Sch_ex}$, there exist $\eta_1\in\mathcal{O}(U)$ and $\eta_2\in\bar{\mathcal{O}}(U)$ such that $\eta=\eta_1+\eta_2$. 
Then, $\del\eta_1= \del\eta=h_1 $ and $\delbar\eta_2= \del \eta=h_2$. 
Hence, $\eps_1(\eta_1, \eta_2)= (h_1, h_2, \eta)$. 
It follows that $\Im \eps_1=\Ker \eps_2$. 

Let $\eps_2 (h_1,h_2,\eta)= (h_1-\del \eta, h_2-\delbar\eta)\in\Im \eps_2 $.
Then, $(\delbar+\del)(h_1-\del \eta, h_2-\delbar\eta)=\delbar(h_1-\del \eta)+\del(h_2-\delbar\eta)=-\delbar\del\eta-\del\delbar\eta=0$. 
Thus, $\Im \eps_2\subset\Ker \delbar+\del$. 
Conversely, we consider $(f,g)\in\Ker\delbar+\del$. 
Since $\delbar f+\del g=0$, both $f$ and $g$ are $\del\delbar$-closed on $U$. 
We apply Lemma $\ref{Sch_ex}$ for $f\in\mathcal{A}^{1,0}(U)$ and $g\in\mathcal{A}^{0,1}(U)$. 
Then, there exist $h_1\in\Omega^1(U)$ and $\eta_1\in\mathcal{A}(U)$ such that $f=h_1+\del\eta_1$, and there exist $h_2\in\bar{\Omega}^1(U)$ and $\eta_2\in\mathcal{A}(U)$ such that $g=h_2+\delbar\eta_2$. 
We find that $\del\delbar(\eta_2-\eta_1)=0$ as $\delbar f+\del g=0$. 
By applying Lemma $\ref{Sch_ex}$ for $\eta_2-\eta_1\in\mathcal{A}(U)$, we get $\tilde{\eta_1}\in\bar{\mathcal{O}}(U)$ and $\tilde{\eta_2}\in\mathcal{O}(U)$ such that $\eta_1-\eta_2=\tilde{\eta_2}-\tilde{\eta_1}$. 
Here, we set $G\coloneqq \eta_1+\tilde{\eta_1}=\eta_2+\tilde{\eta_2}$. 
Then, $\eps_2 (h_1,h_2,-G)=(h_1+\del G, h_2+\delbar G)=(f,g)$. 
Therefore, $\Im \eps_2\supset\Ker \delbar+\del$. \\
\end{proof}

From Lemma $\ref{ex1_1}$, we obtain the exact sequence
\begin{align}\label{fine1_1}
0\to \Ker \eps_2  \overset{i}{\hookrightarrow} \Omega^1\oplus \bar{\Omega}^1 \oplus \mathcal{A} \overset{\eps_2}{\to} \mathcal{A}^{1,0}\oplus\mathcal{A}^{0,1} \overset{\delbar+\del}{\to} \mathcal{A}^{1,1} \overset{\del\delbar}{\to} \mathcal{A}^{2,2} \to 0.
\end{align}
This is a resolution of the sheaf $\Ker \eps_2$ by coherent and fine sheaves. 
Hence, it follows that $H^{1,1}_A(X) \cong H^2(\mathscr{V}, \Ker\eps_2)$ for the Stain covering $\mathscr{V}$ as in \S \ref{sec_Aeppli01}. 

Subsequently, we observe another resolution of $\Ker\eps_2$. 
\begin{lem}\label{another_sq1_1}
The sheaf sequence 
\begin{align*}
0 \to \Ker\eps_2 \overset{i}{\hookrightarrow} \Omega^1\oplus \bar{\Omega}^1 \oplus \mathcal{G} \overset{\tilde{\eps_2}}{\to} \mathcal{G}^{1,0}\oplus\mathcal{G}^{0,1} \overset{\delbar+\del}{\to} \mathcal{G}^{1,1} \to 0.
\end{align*}
is exact, where $\tilde{\eps_2}; (h_1,h_2,\eta)\mapsto (h_1-\del \eta, h_2-\delbar\eta)$. 
\end{lem}
\begin{proof}
Let $U\subset X$ be a sufficiently small open ball. 
It is clear that $\Ker\eps_2 = \Ker \tilde{\eps_2}$. 
It suffices to show that $\Im \tilde{\eps_2}=\Ker \delbar+\del$ on $U$ and $\delbar+\del: \mathcal{G}^{1,0}(U)\oplus\mathcal{G}^{0,1}(U) \to \mathcal{G}^{1,1}(U)$ is surjective. 

Clearly, $\Im\tilde{\iota_2}\subset\Ker \delbar+\del$. 
Conversely, we take $(f,g)\in \Ker \delbar+\del$. 
For this pair $(f,g)$, we can choose $h_1\in\Omega^1(U), h_2\in\bar{\Omega^1}(U)$, and $G\in\mathcal{A}(U)$ as in the proof of Lemma $\ref{ex1_1}$. 
That is, we define $G\coloneqq \eta_1+\tilde{\eta_1}$ such that $\tilde{\eta_1}\in \bar{\mathcal{O}}(U)$ and $\eps_2 (h_1, h_2, -G)= (f,g)$.
We prove that $-G$ is an element of $\mathcal{G}(U)$. 
Now, $h_1\in\Omega^1(U)$ can be written as $h_1=f_1 dz_1+ f_2 dz_2$ using holomorphic functions $f_1$ and $f_2$ on $U$. 
Then, 
\begin{align*}
f=\left(f_1+\frac{\del\eta_1}{\del z_1}\right)dz_1+ \left(f_2+\frac{\del\eta_1}{\del z_2}\right)dz_2 \in\mathcal{F}(U)
\end{align*}
is $\delbar_{z_2}$-closed. 
Hence, we have $\del_{z_2}\delbar_{z_2}\eta_1=0$. 
Therefore, $\del_{z_2}\delbar_{z_2}G=\del_{z_2}\delbar_{z_2}\eta_1=0$. 
That is, $-G\in\mathcal{G}(U)$.

Take $f=f_1dz_1\wedge d\bar{z}_1+ f_2 dz_1\wedge d\bar{z}_2+ f_3dz_2\wedge d\bar{z}_1\in\mathcal{G}^{1,1}(U)$, where $f_1\in\mathcal{G}(U), f_2\in\bar{\mathcal{F}}(U)$, and $f_3\in\mathcal{F}(U)$. 
Since $f_1$ is $\del_{z_2}\delbar_{z_2}$-closed, there exist $h_1\in\mathcal{F}(U)$ and $h_2\in\bar{\mathcal{F}}(U)$ such that $f_1=h_1+h_2$. 
We set 
\begin{align*}
F_1\coloneqq h_1dz_1\wedge d\bar{z}_1+ f_3dz_2\wedge d\bar{z}_1, \\
F_2\coloneqq h_2dz_1\wedge d\bar{z}_1+ f_2 dz_1\wedge d\bar{z}_2.
\end{align*}
Then $F_1$ is $\delbar$-closed and $F_2$ is $\del$-closed. 
Therefore we obtain $\tilde{F_1}\in\mathcal{A}^{1,0}(U)$ and $\tilde{F_2}\in\mathcal{A}^{0,1}(U)$ such that $F_1=\delbar\tilde{F_1}$ and $F_2=\del\tilde{F_2}$, respectively. 
By the independence of alternating multilinear form, we find that $\tilde{F_1}$ is $\delbar_{z_2}$-closed and $\tilde{F_2}$ is $\del_{z_2}$-closed. 
Thus, $(\delbar+\del)(\tilde{F_1}, \tilde{F_2})=f$. 
\end{proof}

We obtain the following theorem from Lemma $\ref{another_sq1_1}$ and Lemma $\ref{G_vani}$ in a similar manner as in \S $\ref{sec_Aeppli00}$. 
\begin{prop}\label{cong1_1}
\begin{align*}
H^{1,1}_A (X) \cong \frac{ \mathcal{G}^{1,1}(X) }{\Im (\delbar : \mathcal{G}^{1,0} (X) \to \mathcal{G}^{1,1}(X))+\Im (\del : \mathcal{G}^{0,1}(X) \to \mathcal{G}^{1,1}(X))}.
\end{align*}
\end{prop}

By Proposition $\ref{cong1_1}$, it suffices to observe an element $w=F_1 dz_1\wedge d\bar{z}_1 + F_2 dz_2\wedge d\bar{z}_1 + F_3 dz_1\wedge d\bar{z}_2$ in $\mathcal{G}^{1,1}(X)$ instead of $\mathcal{A}^{1,1}(X)$, where $F_1\in\mathcal{G}(X), F_2\in\mathcal{F}(X)$ and $F_3\in\bar{\mathcal{F}}(X)$. 
From the Fourier series expansion as $(\ref{Fourier})$, we have
\begin{align*}
F_i(z_1, z_2) = \sum_{\sigma\in\Z^3} a_i^\sigma (\Im z_2) \exp{\langle \sigma, t' \rangle},
\end{align*}
for each $F_i$. 
Note that the notations $F_i$ and $a_i^\sigma$ are defined here specifically for this subsection and may differ from those used in \S \ref{sec_Aeppli01}. 

Assume that there exist  $\psi_1=\psi_1^1 dz_1+ \psi_1^2 dz_2 \in\mathcal{G}^{1,0}(X)$ and $\psi_2=\psi_2^1 d\bar{z}_1+ \psi_2^2 d\bar{z}_2 \in \mathcal{G}^{0,1}(X)$ such that $\delbar\psi_1+ \del\psi_2= w$. 
Then, using $C^\infty$ functions $b_{i,j}^\sigma$, we get Fourier series expansions
\begin{align*}
\psi_i^j (z_1,z_2) = \sum_{\sigma\in\Z^3} b_{i,j}^\sigma (\Im z_2) \exp{\langle \sigma, t' \rangle}.
\end{align*}
Since $\psi_1^2$ is holomorphic along the $z_2$-direction and $\psi_2^2$ is anti-holomorphic along the $z_2$-direction, 
\begin{align*}
\delbar \psi_1 = \sum_{j=1,2} \left(\sum_{\sigma\in\Z^3} b_{1,j}^\sigma A^\sigma \exp{\langle \sigma, t' \rangle} \right) d\bar{z}_1\wedge dz_j + \sum_{\sigma\in\Z^3} \left( \frac{\del b_{1,1}^\sigma}{\del \bar{z}_2}+ B^\sigma b_{1,1}^\sigma \right) \exp{\langle \sigma, t' \rangle} d\bar{z}_2\wedge dz_1, \\
\del \psi_2= \sum_{j=1,2} \left( \sum_{\sigma\in\Z^3} b_{2,j}^\sigma(- \bar{A^\sigma}) \exp{\langle \sigma, t' \rangle} \right) dz_1\wedge d\bar{z}_j + \sum_{\sigma\in\Z^3} \left( \frac{\del b_{2,1}^\sigma}{\del z_2}+ B^\sigma b_{2,1}^\sigma \right) \exp{\langle \sigma, t' \rangle} dz_2\wedge d\bar{z}_1.
\end{align*}

The comparison of coefficients in $\delbar\psi_1+ \del\psi_2= w$ gives the following relations for each $\sigma\in\Z^3$:
\begin{align}\label{relation11}
\left\{
\begin{array}{lll}
a_1^\sigma = -b_{2,1}^\sigma \bar{A^\sigma} - b_{1,1}^\sigma A^\sigma, \\
a_2^\sigma = -b_{1,2}^\sigma A^\sigma +\left(\frac{\del b_{2,1}^\sigma}{\del z_2}+ B^\sigma b_{2,1}^\sigma\right),\\
a_3^\sigma = -\left(\frac{\del b_{1,1}^\sigma}{\del \bar{z}_2}+ B^\sigma b_{1,1}^\sigma\right)-b_{2,2}^\sigma \bar{A^\sigma}.
\end{array}
\right.
\end{align}
As $F_1\in\mathcal{G}(X)$, we find that
\begin{align*}
0=\del_{z_2}\delbar_{z_2} (a_1^\sigma \exp{\langle \sigma,t' \rangle})
=\left( \frac{1}{4} \cdot \frac{d^2 a_1^\sigma}{dt_4^2} + (B^\sigma)^2 a_1^\sigma \right)d\bar{z}_2\wedge dz_2.
\end{align*}
The general solution to this differential equation can be written as follows using constants $C_1^\sigma$ and $C_2^\sigma$:
\begin{align}\label{solution_a1}
a_1^\sigma=
\left\{
\begin{array}{ll}
C_1^\sigma\exp{(2\pi \sigma_2 t_4)}+ C_2^\sigma \exp{(-2\pi \sigma_2t_4)} & (\sigma_2\neq 0) \\
C_2^\sigma t_4+ C_1^\sigma & (\sigma_2=0)
\end{array}.
\right.
\end{align}

Based on the preceding discussion, we formally construct a $(\delbar+\del)$-solution with $\delbar\psi_1+ \del\psi_2= w$. 
In order to satisfy the relation $(\ref{relation11})$, for $\sigma\in\Z^3\setminus\{0\}$, we set
\begin{align*}
b_{1,1}^\sigma &\coloneqq
\left\{
\begin{array}{ll}
-\frac{C_1^\sigma}{A^\sigma} \exp{(-2\pi\sigma_2 t_4)} & (\sigma_2\neq 0) \\[1ex] \\
\frac{1}{2\sqrt{-1}A^\sigma}\left(C_2^\sigma \bar{z}_2 -2\sqrt{-1}C_1^\sigma \right) &(\sigma_2=0)
\end{array}
\right. ,\\
b_{2,1}^\sigma &\coloneqq 
\left\{
\begin{array}{ll}
-\frac{C_2^\sigma}{\bar{A^\sigma}} \exp{(-2\pi\sigma_2 t_4)} & (\sigma_2\neq 0) \\[1ex] \\
-\frac{1}{2\sqrt{-1}\ \bar{A^\sigma}}\left(C_2^\sigma z_2 \right) & (\sigma_2=0)
\end{array}
\right. ,\\
b_{1,2}^\sigma &\coloneqq -\frac{1}{A^\sigma}\left(\frac{\del b_{2,1}^\sigma}{\del z_2}+B^\sigma b_{2,1}^\sigma- a_2^\sigma \right), \\
b_{2,2}^\sigma &\coloneqq -\frac{1}{\bar{A^\sigma}}\left( \frac{\del b_{1,1}^\sigma}{\del \bar{z_2}}+B^\sigma b_{1,1}^\sigma+ a_3^\sigma \right).
\end{align*}
Thus, using $b_{i,j}^\sigma$ we defined as above, we obtain a formal solution
\begin{align*}
\tilde{\psi}_1&\coloneqq \left(\sum_{\sigma\neq0} b_{1,1}^\sigma (t_4) \exp{\langle \sigma, t' \rangle}\right)dz_1+ \left(\sum_{\sigma\neq0} b_{1,2}^\sigma (t_4) \exp{\langle \sigma, t' \rangle}\right)dz_2, \\
\tilde{\psi}_2&\coloneqq \left(\sum_{\sigma\neq0} b_{2,1}^\sigma (t_4) \exp{\langle \sigma, t' \rangle}\right)d\bar{z}_1+ \left(\sum_{\sigma\neq0} b_{2,2}^\sigma (t_4) \exp{\langle \sigma, t' \rangle}\right)d\bar{z}_2.
\end{align*}
We conclude that the following equality holds formally:
\begin{align*}
w=\delbar\tilde{\psi}_1+\del\tilde{\psi}_2+ a_1^0(t_4)dz_1\wedge d\bar{z}_1+ a_2^0(t_4)dz_1\wedge d\bar{z}_2+a_3^0(t_4)dz_2\wedge d\bar{z}_1.
\end{align*}
It is clear that $a_i^0(\Im z_2)$ are uniquely determined by $w$ due to our construction above.
Now, we can regard $a_2^0(t_4)$ as an anti-holomorphic function on $X$. 
It follows that $a_2^0(t_4)$ is constant. 
Similarly, since $a_3^0(t_4)$ is holomorphic, it is constant. 
Moreover, $a_1^0(t_4)$ can be written as $C_2^0 t_4+ C_1^0$ from $(\ref{solution_a1})$. 
If $\tilde{\psi}_1$ (resp. $\tilde{\psi}_2$) converges, we find that $\tilde{\psi}_1\in\mathcal{G}^{1,0}(X)$ (resp. $\tilde{\psi}_2\in\mathcal{G}^{0,1}(X)$) by simple calculations. 
Therefore we obtain the formal representation
\begin{align}\label{repre1_1}
w=\delbar\tilde{\psi}_1+\del\tilde{\psi}_2+ (C_2 t_4+ C_1)dz_1\wedge d\bar{z}_1+ C_3dz_1\wedge d\bar{z}_2+C_4dz_2\wedge d\bar{z}_1\nonumber\\
=\delbar (\tilde{\psi}_1+ 2C_3\i\Im z_2 dz_1) +\del (\tilde{\psi}_2 +2C_4\i\Im z_2 d\bar{z}_1)\\
+(C_2 t_4+ C_1)dz_1\wedge d\bar{z}_1\nonumber,
\end{align}
where the constans $C_i$ are uniquely determined.
Note that $2C_3\i\Im z_2 dz_1 \in\mathcal{G}^{1,0}(X)$ and $2C_4\i\Im z_2 d\bar{z}_1 \in \mathcal{G}^{0,1}(X)$.

\end{subsection}
\end{section}
\begin{section}{Proof of Theorem $\ref{Aeppli}$}\label{pf_12} 
We have proven that $h^{0,0}_A=1$ in \S $\ref{sec_Aeppli00}$. 
Subsequently, we identify $H^{2,q}_A(X)$ and $H^{p,2}_A(X)$. 
By the definition, 
\begin{align*}
H^{2,0}_A(X) &= \Ker \del\delbar\cap\mathcal{A}^{2,0}(X)/\del \mathcal{A}^{1,0}(X)+ \Omega^2(X)\\
&= \Ker \del\cap\mathcal{A}^{2,0}(X)/\del \mathcal{A}^{1,0}(X)+ \Omega^2(X), \\
H^{2,q}_A(X) &= \Ker \del\delbar\cap\mathcal{A}^{2,q}(X)/\del \mathcal{A}^{1,q}(X)+ \delbar \mathcal{A}^{2,q-1}(X)\\
&= \Ker \del\cap\mathcal{A}^{2,q}(X)/\del \mathcal{A}^{1,q}(X)+ \delbar \mathcal{A}^{2,q-1}(X) 
\end{align*}
for $q\geq1$. 
Then, we find that $h^{2,q}_A\leq h^{2,q}_\del$ for any $q\in\Z_{\geq0}$. 
Hence, we have $h^{2,q}_A=0$ from Remark $\ref{del-cohom}$. 
Similarly, we find that its conjugate $h^{p,2}_A=0$ from Corollary $\ref{dol-cohom}$. 

Next, we consider the Hausdorff completion $\tilde{H}^{1,1}_A(X)$ of the Aeppli cohomology of type $(1,1)$. 
Denote by $T$ the continuous linear map $\delbar+\del: \mathcal{A}^{(0,0)+1}(X) \to \Ker\del\delbar \cap\mathcal{A}^{1,1}(X)$. 
Note that $\tilde{H}^{1,1}_A(X)=\Ker\del\delbar\cap\mathcal{A}^{1,1}(X) /\bar{\Im T}$, where $\bar{\Im T}$ is the closure of $\Im T$. 
Let $E$ be the two-dimensional topological vector space generated by $\{ dz_1\wedge d\bar{z}_1, t_4dz_1\wedge d\bar{z}_1\}$.
We define a continuous linear map
\begin{align*}
\tilde{T}: \mathcal{A}^{(0,0)+1}(X)\oplus E \to \mathcal{A}^{1,1}(X); (f,g)\mapsto T(f)+g.
\end{align*}
Then, $\Im \tilde{T}$ is dense in $\Ker \del\delbar\cap \mathcal{A}^{1,1}(X)$ from the observation in \S \ref{sec_Aeppli11}. 
In fact, $\tilde{\psi}_1$ and $\tilde{\psi}_2$ as in $(\ref{repre1_1})$ are formal series, but their finite partial sums are in $\mathcal{G}^{1,0}(X)$ and in $\mathcal{G}^{0,1}(X)$, respectively. 
Hence, the closure $\bar{\Im T}\oplus E$ of $\Im \tilde{T}$ is equal to $\Ker \del\delbar (X)$. 
Therefore, the Hausdorff completion $\tilde{H}^{1,1}_A(X)$ is equal to $E$, that is, $\tilde{h}^{1,1}_A=2$. 
In a similar manner, we can prove that $\tilde{h}^{0,1}_A=1$ and $\tilde{h}^{1,0}_A=1$.

\end{section}
\begin{section}{Appendix}
\begin{subsection}{The $\del\delbar$-lemma on compact complex manifolds}\label{ddbar_lemma}
By Angella--Tomassini \cite{AT} and Angella--Tardini \cite{ATa}, the validity of $\del\delbar$-lemma was characterized by the dimensions of some cohomologies for compact complex manifolds. 
\begin{thm}[{\cite[Theorem $3.1$]{ATa}}]
A compact complex manifold $X$ satisfies the $\del\delbar$-lemma if and only if 
\begin{align}\label{equiv_deldelbar}
\sum_{p+q=k} (h^{p,q}_{BC}+h^{p,q}_A)=2b_k
\end{align}
for any $k\in\Z_{>0}$, where $b_k$ are the $k$-th betti numbers.
\qed
\end{thm}

In \cite{AT}, they defined the $k$-th non-$\del\delbar$-degree for compact complex manifolds. 
\begin{defi}[{\cite[Theorem A]{AT} (See also \cite[Theorem $2.1$]{Ange}.)}]
Let $X$ be a compact complex manifold. 
The {\it$k$-th non-$\del\delbar$-degree} is defined as
\begin{align*}
\Delta^k(X) \coloneqq h^k_{BC} +h^k_A -2b_k
\end{align*}
for any $k\in \Z_{\geq0}$, where $h^k_{\bullet}\coloneqq \sum_{p+q=k} h^{p,q}_{\bullet}$ and $b_k$ is $k$-th betti number.
\end{defi}
Then, all $\Delta^k$ are non-negative integers since the Fr\"{o}licher-type inequality of Bott--Chern and Aeppli cohomologies holds for compact complex manifolds.
Accordingly, we generalize the non-$\del\delbar$-degree to theta toroidal groups. 
If $X=\C^2/\Lambda_{\tau,p,q}$ and all the cohomology groups $H^{p,q}_A(X)$ are Hausdorff, then
\begin{align*}
\Delta^0=0,\ \Delta^1=0,\ \Delta^2=1,\ \Delta^3=0,\ \Delta^4=0.
\end{align*}

Consequently, theta toroidal groups possess some properties similar to compact manifolds, but we conjecture that the equality $(\ref{equiv_deldelbar})$ does not hold from Theorem $\ref{BCcohom}$ and Theorem $\ref{Aeppli}$. 

\begin{remark}\label{standard_def}
Under the standard definitions of the Bott--Chern and Aeppli cohomologies, the Bott--Chern cohomology groups coincides with the one given in Theorem \ref{BCcohom}, whereas the Aeppli numbers are
\begin{align*}
\begin{array}{ccccc}
     &  & h^{0,0}_{A}=2 &  &  \\
     & \tilde{h}^{1,0}_{A}=2 &  & \tilde{h}^{0,1}_{A}=2 &  \\
    h^{2,0}_{A}=0 &  & \tilde{h}^{1,1}_A=2 &  & h^{0,2}_{A}=0 \\
     & h^{2,1}_{A}=0 &  & h^{1,2}_{A}=0 &  \\
     &  & \phantom{\;.} h^{2,2}_{A}=0 \;. &  & 
\end{array}
\end{align*}
Consequently, in this case as well, some of the non-$\del\delbar$-degree are non-zero: 
\begin{align*}
\Delta^0=1,\ \Delta^1=2,\ \Delta^2=1,\ \Delta^3=0,\ \Delta^4=0.
\end{align*}
\end{remark}

\end{subsection}
\begin{subsection}{Third cohomology}\label{thirdcohom}
Let $X$ be a complex manifold. 
The cohomology $H^{(p,q)+1}_T(X)$ of type $(p,q)$ of $X$ is the cohomology of the complex 
\begin{align*}
\mathcal{A}^{p,q}(X) \overset{d}{\to} \mathcal{A}^{(p,q)+1}(X) \overset{\delbar+\del}{\to} \mathcal{A}^{p+1,q+1}(X). 
\end{align*}
According to \cite{MT}, we call $H^{(p,q)+1}_T(X)$ the third cohomology (for two types of cohomology, Bott--Chern and Aeppli), namely, 
\begin{align*}
H^{(p,q)+1}_T(X)\coloneqq \Ker (\delbar+\del) \cap \mathcal{A}^{(p,q)+1} / \Im d.
\end{align*}

In the following, we consider the third cohomology for a two-dimensional theta toroidal group $X= \C^2/\Lambda_{\tau,p,q}$. 
Some of these cohomology groups can be easily identified. 

By the definition, it is clear that $H^{(2,1)+1}_T(X)=H^{(1,2)+1}_T(X)=H^4(X, \C)=0$. 
In generally, the sheaf sequence
\begin{align*}
\cdots \overset{\del\delbar}{\to} \mathcal{A}^{p-1,q-1} \overset{d}{\to} \mathcal{A}^{(p-1,q-1)+1} \overset{\del +\delbar}{\to} \mathcal{A}^{p,q} \overset{\del\delbar}{\to} \mathcal{A}^{p+1,q+1} \overset{d}{\to} \cdots
\end{align*}
is not exact. 
However, the sequece
\begin{align*}
0 \to \mathcal{H} \overset{i}{\hookrightarrow} \mathcal{A}^{0,0} \overset{\del\delbar}{\to} \mathcal{A}^{1,1} \overset{d}{\to} \mathcal{A}^{(1,1)+1} \overset{\delbar+\del}{\to} \mathcal{A}^{2,2} \to 0
\end{align*}
is exact. 
In fact, we can confirm the exactness from Lemma $\ref{Sch_ex}$. 
From the fine resolution of the sheaf $\mathcal{H}$, we obtain $H^2(X, \mathcal{H}) \cong H^{(1,1)+1}_T (X)$. 
By observing the long exact sequence 
\begin{align*}
0 \to H^0(X, \C) \to H^0(X, \mathcal{O}\oplus\bar{\mathcal{O}}) \to H^0 (X, \mathcal{H}) \to H^1(X, \C) \to H^1(X, \mathcal{O}\oplus\bar{\mathcal{O}}) \label{longex_BC0_0} \\
\to H^1(X, \mathcal{H}) \to H^2(X, \C) \to H^2 (X, \mathcal{O}\oplus\bar{\mathcal{O}})=0 \to \cdots, \nonumber
\end{align*}
we find that $H^2(X, \mathcal{H}) \cong H^3(X, \C)$. 
Therefore, $h_T^{(1,1)+1}\coloneqq \dim H^{(1,1)+1}_T (X) =1$.

\end{subsection}
\end{section}


\end{document}